\documentclass{article}
\usepackage{graphicx} % Required for inserting images
\usepackage[normalem]{ulem}
\usepackage{amsmath, amssymb,amsthm}
\usepackage{soul}
\usepackage{pgf,tikz,calc}
\usepackage{enumitem}
\usepackage[normalem]{ulem}
\usepackage[colorlinks=true, linkcolor=blue]{hyperref}
\usepackage{verbatim}
\newtheorem{theorem}{Theorem}
\newtheorem{proposition}[theorem]{Proposition}
\newtheorem{corollary}[theorem]{Corollary}

\newtheorem{observation}[theorem]{Observation}
\newtheorem{claim}{Claim}

\newtheorem{open}{Open problem}
\newtheorem{conj}{Conjecture}
\usetikzlibrary{fit,backgrounds}

\usepackage{amsthm}
\theoremstyle{definition}
\newtheorem{definition}{Definition}

\newcommand{\smallqed}{{\tiny ($\Box$)}}
\newcommand{\mc}{\overline{\chi}_{\geqslant}}
\newcommand{\cg}{{\rm C}_{\geqslant}}

\newcommand{\cP}{{\cal P}}

\newcommand{\cH}{{\cal H}}
\newcommand{\cE}{{\cal E}}

\usepackage[top=1in, bottom=1.25in, left=1.35in, right=1.35in]{geometry}

\setlist[description]{style=nextline}

\newcommand{\hf}[1]{ \textcolor{orange}{HF: #1}}

\usepackage[dvipsnames]{xcolor}

\DeclareMathOperator{\dist}{dist}

\begin{document}

\title{Majority C-coloring of graphs}
\author{
Csilla Bujt\'as $^{a,b,}$\thanks{Email: \texttt{csilla.bujtas@fmf.uni-lj.si}}
\and
Magda Dettlaff $^{c,}$\thanks{Email: \texttt{magda.dettlaff@ug.edu.pl}}
\and
Hanna Furma\'nczyk $^{c,}$\thanks{Email: \texttt{hanna.furmanczyk@ug.edu.pl}}
\and
Aleksandra Laskowska $^{c,}$\thanks{Email: \texttt{aleksandra.laskowska@ug.edu.pl}}
}
\date{}
\maketitle

\begin{center}
$^a$ Faculty of Mathematics and Physics, University of Ljubljana, Slovenia\\
\medskip

$^b$ Institute of Mathematics, Physics and Mechanics, Ljubljana, Slovenia\\
\medskip

$^c$ Faculty of Mathematics, Physics and Informatics, University of Gda\'nsk, Poland\\
\medskip
\end{center}
\maketitle

\begin{abstract}
Inspired by the majority colorings and C-colorings, we introduce and study the majority C-coloring of graphs. In such a vertex coloring, every vertex shares its color with at least half of its neighbors. The maximum number of colors that can be used in a majority C-coloring of a graph $G$ is called the majority C-chromatic number and denoted by $\mc(G)$. 

An upper bound on $\mc(G)$ is proved in terms of the order, minimum, and maximum degree. Its sharpness is demonstrated by several results over different graph classes. In particular, $\mc(P_n^k)= \mc(C_n^k)= \lfloor n/(k+1)\rfloor$ is true for the $k$-th power of a path and a cycle if $n \ge k+1$. Further, $\mc(G) = (n-d)/3$ holds if $G$ is a $(\mbox{claw}, K_4)$-free cubic graph and contains $d$ diamonds. %claw-free cubic graph on $n \ge 6$ vertices and contains $d$ diamonds. 
It is further shown that the majority C-chromatic number is not monotone under edge deletion. In fact, both the lower and upper bounds are sharp in the inequality chain $\mc(G)-2 \leq \mc(G-e) \leq \mc(G) +1$. The minimum and maximum number of edges in an $n$-vertex graph $G$ with $\mc(G)=k$ are determined for every $n$ and $k$. It is also pointed out that the classical chromatic number $\chi(G)$ and $\mc(G)$ are incomparable, and the difference $\mc(G)-\chi(G)$ can take any positive or negative integer. On the other hand,  $\mc(G)+\chi(G) \leq n+1$ holds for every graph $G$ of order $n$. The decision problem of whether $\mc(G) \ge k$ holds is NP-complete for every fixed $k\ge 2$. In contrast, some sufficient conditions for $\mc(G) \ge 2$ are proved, and a linear-time algorithm is presented that determines $\mc(T)$ if $T$ is a tree. 
\end{abstract}
\medskip

\noindent
{\bf Keywords:} majority coloring; C-coloring; satisfactory partition; majority C-coloring; power of a path; power of a  cycle; cubic graph; chromatic number; algorithm on trees.
\medskip

\noindent
{\bf AMS Subj.\ Class.\ (2020):}  05C15, 05C85

%%%%%%%%%%%%%%%%%%%%%%%%%%%%%%%%%
\section{Introduction}

The classical vertex coloring of graphs requires that different colors be assigned to adjacent vertices, and a central question asks for the chromatic number, which is the minimum number of colors in such a coloring. Over the long and celebrated history of chromatic graph theory, several variants, both strengthenings and relaxations, of the coloring constraint were studied. However, the focus remained on the problems where the goal is to use as few colors as possible. In this paper, we explore a different direction when the constraint asks for (a certain number of) neighbors with the same color and we are interested in maximizing the number of colors.  
\medskip

We define the \emph{majority C-coloring} (shortly, \emph{$\cg$-coloring})  of a simple graph $G$ as a vertex coloring so that each vertex has a \textbf{c}ommon color with at least half of its neighbors. A coloring that assigns the same color to all vertices of a graph $G$ is always a majority C-coloring, implying that the minimum number of colors is always $1$. Here, the challenging problem is to determine the possible maximum number of colors that can be used in a majority C-coloring of a given graph $G$. We call this maximum number the \emph{majority C-chromatic number} (shortly, \emph{$\cg$-chromatic number}) of $G$ and denote it by $\mc(G)$. A formal definition is given in Section~\ref{sec:pre}. 

Considering, for example,  path and cycle graphs, it is easy to observe that the $\cg$-chromatic number can be arbitrarily high over these graph classes. In contrast, complete graphs and a large subclass of complete bipartite graphs provide examples with $\mc(G)=1$ for arbitrarily large order.
For examples, see Fig.~\ref{fig:ex_substar_cubic}.

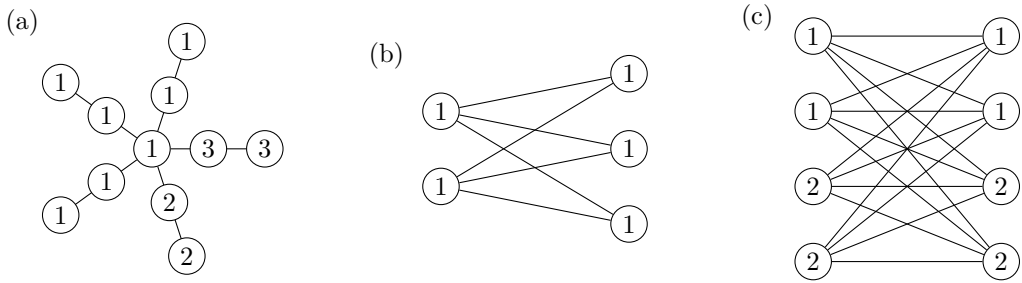
\begin{figure}[htb] \label{fig:ex_substar_cubic}
\centering
\begin{minipage}{0.3\textwidth}
\raggedright \makebox[0cm][l]{ (a)}\hspace*{0.6cm} \\[-0.6em]
\centering
 \begin{tikzpicture}[scale=1.0, every node/.style={circle, draw, fill=white, inner sep=2pt}]
        \def\n{5}
       \def\radius{1.5}
    
        % Root
        \node (r) at (0,0) {1};
        
        % Subdivided arms
        \foreach \i in {1,...,\n} {
            \coordinate (u\i) at (\i*360/\n:\radius/2); % subdivision vertex
           \coordinate (v\i) at (\i*360/\n:\radius);   % leaf vertex
            \node (ui\i) at (u\i) {};  % empty node for now
            \node (vi\i) at (v\i) {};  % empty node for now
            \draw (r) -- (ui\i) -- (vi\i);
        }

        % Coloring numbers (centered)
        \node at (ui1) {1};
        \node at (vi1) {1};
        \node at (ui2) {1};
        \node at (vi2) {1};
        \node at (ui3) {1};
        \node at (vi3) {1};
        \node at (ui4) {2};
        \node at (vi4) {2};
        \node at (ui5) {3};
        \node at (vi5) {3};
    \end{tikzpicture}

    \end{minipage}
\hspace{0.02\textwidth}
\begin{minipage}{0.3\textwidth}
\raggedright \makebox[0cm][l]{(b)}\hspace*{0.6cm}
\\[-0.6em] % mały pionowy minus, opcjonalnie
\centering
\begin{tikzpicture}[
scale=1,
vertex/.style={circle, draw, inner sep=2pt},
lab/.style={font=\scriptsize}
]

% LEFT side
\node[vertex] (l1) at (0,2) {1};
\node[vertex] (l2) at (0,1) {1};
%\node[vertex] (l3) at (0,0) {1};
%\node[vertex] (l4) at (0,1) {1};
%\node[vertex] (l5) at (0,0) {1};

% RIGHT side 
\node[vertex] (r1) at (2.5,0.5) {1};
\node[vertex] (r2) at (2.5,1.5) {1};
\node[vertex] (r3) at (2.5,2.5) {1};
% edges
\foreach \i in {1,2}
  \foreach \j in {1,2,3}
    \draw (l\i) -- (r\j);
\end{tikzpicture}
\end{minipage}
\hspace{0.02\textwidth}
\begin{minipage}[c]{0.3\textwidth}
\raggedright \makebox[0cm][l]{(c)}\hspace*{0.6cm}
\\[-0.6em] % mały pionowy minus, opcjonalnie
\centering
\begin{tikzpicture}[
scale=1,
vertex/.style={circle, draw, inner sep=2pt},
lab/.style={font=\scriptsize}
]
% LEFT side
\node[vertex] (l1) at (0,3) {1};
\node[vertex] (l2) at (0,2) {1};
\node[vertex] (l3) at (0,1) {2};
\node[vertex] (l4) at (0,0) {2};

% RIGHT side
\node[vertex] (r1) at (2.5,3) {1};
\node[vertex] (r2) at (2.5,2) {1};
\node[vertex] (r3) at (2.5,1) {2};
\node[vertex] (r4) at (2.5,0) {2};

% edges
\foreach \i in {1,2,3,4}
  \foreach \j in {1,2,3,4}
    \draw (l\i) -- (r\j);

\end{tikzpicture}
\end{minipage}
\caption{$\cg$-coloring of (a) subdivided star $S(S_5)$, (b) complete bipartite graph $K_{2,3}$, and (c) complete bipartite graph $K_{4,4}$, using the maximum number of colors.}
\label{fig:ex_substar_cubic}
\end{figure}

%\cs{Recall from~\cite{GeKob-1998} that the sum satisfactory $k$-partition of a graph $G=(V,E)$ is a (nontrivial) partition  $\cP=\{V_1,\dots, V_k\}$ of $V$, where each vertex $v \in V_i$ has at least half of its neighbors in the same class $V_i$, for $i \in [k]$. 
%It was proved in~\cite{BTV-06} that this problem is NP-complete for every $k \ge 2$.} Equivalently, with the present terminology, the decision problem of $\mc(G) \ge k$ is NP-complete for every fixed integer $k \ge 2$. \cs{Or we may omit this paragraph and discuss it in the motivation/satisfactory partition part. ?? If not, we should remove the definition of the sum satisfactory $k$-partition from there.}
\medskip

By definition, the color classes of a $\cg$-coloring may represent communities in biological or social networks. In general, the notion of majority C-coloring is closely connected to clustering problems. Further applications will be discussed in more detail in Section~\ref{sec:final}.

\paragraph{Standard definitions.} 
Throughout the paper, we consider only simple undirected graphs and use the standard graph-theoretic notation according to~\cite{West-book}.
For a positive integer $k$, the symbol $[k]$ stands for the set $\{1,2,\ldots,k\}$.
When a graph is explicitly denoted as $G=(V,E)$, we write $V$ and $E$ for its vertex set and edge set;
otherwise, we use the notation $V(G)$ and $E(G)$ to emphasize the dependence on the underlying graph.
For a vertex $v$ of a graph $G$, we denote by $N_G(v)$ the {\it open neighborhood} of $v$  which is the set of its neighbors. Thus $\deg_G(v)=|N_G(v)|$. A vertex of degree $1$ is a {\it leaf}. %The \emph{minimum degree} of a graph $G$ is $\min_{v\in V(G)}\deg_G(v)$ and it is denoted by $\delta(G)$ while the \emph{maximum degree} of a graph $G$ is $\max_{v\in V(G)}\deg_G(v)$ and it is denoted by $\Delta(G)$. 
We will refer to the minimum degree, maximum degree, clique number, and girth of a graph $G$ as $\delta(G)$, $\Delta(G)$, $\omega(G)$, and $g(G)$, respectively.

For $u,v\in V(G)$, the {\it distance} between $u$ and $v$ in $G$, denoted by $\mathrm{dist}_{G}(u,v)$, is the length of the shortest path between them in $G$. When the graph $G$ is known from the context we will write $N(v)$, $\deg(v)$, and $\dist(u,v)$ instead of $N_G(v)$, $\deg_G(v)$, and $\dist_G(u,v)$, respectively. 

For a positive integer $k$, the $k$-{\it th power of a graph} $G$ is the graph $G^{k}$
defined on the vertex set $V(G)$ such that two vertices $u$ and $v$ are adjacent in $G^{k}$
if and only if $\mathrm{dist}_{G}(u,v)\le k$. Recall that the \emph{diamond} is the graph  $K_4-e$ and  the \emph{claw} is the graph $K_{1,3}$.
Finally, a {\it spanning subgraph} of a graph $G$ is a subgraph that contains all vertices of $G$.

\paragraph{Our results and the structure of the paper.} 
In this paper, we start the study of majority C-coloring by considering it over specific graph classes and proving some basic properties. In Section~\ref{sec:pre}, we establish two sharp upper bounds on $\mc(G)$ and prove that $\mc(G)$ is not monotone under taking spanning subgraphs. Further, the $\cg$-chromatic number of the spanning subgraph $G-e$, obtained by deleting an edge $e$ from $G$, satisfies 
\begin{equation*}
        \mc(G)-2 \leq \mc(G-e) \leq \mc(G) +1,
        \label{ineq:del_edge}
    \end{equation*} 
    where both the lower and upper bounds are sharp.
    Section~\ref{sec:graph_classes} is devoted to two graph classes, both demonstrating that $\mc(G)$ can take an arbitrary value over the class. 
    In particular, for the powers of paths and cycles, the same formula $$\mc(P_n^k)=\mc(C_n^k)= \Bigl\lfloor \frac{n}{k+1} \Bigr\rfloor$$
    is proved for every pair $(n,k)$ with $n \ge k+2$. If $G$ is a cubic graph of order $n$ that contains $d$ diamonds, the upper bound $\mc(G) \leq \frac{n-d}{3}$ is valid and holds with equality if $G$ is also $(\mbox{claw}, K_4)$-free. In Section~\ref{sec:min-max}, the maximum and minimum number of edges in an $n$-vertex graph $G$ with $\mc(G)=k$ is identified for every $n \ge k \ge 2$. The classical chromatic number $\chi(G)$ and the $\cg$-chromatic number $\mc(G)$ are incomparable, also in the sense that their difference and sum can be arbitrary. On the other hand, if $n$ is the order of $G$, then
    $n+1$ is a tight upper bound on  $\chi(G)+\mc(G)$.

    As a consequence of earlier results, we conclude that deciding whether $\mc(G) \ge k$ is an NP-complete problem for every fixed $k \ge 2$. In contrast, we design a linear-time algorithm in Section~\ref{sec:sec-tree} that determines the exact value of $\mc(T)$ for a tree $T$. In Section~\ref{sec:sufficient} we present sufficient conditions on $G$ to satisfy $\mc(G) \ge 2$.  Finally, in Section~\ref{sec:final}, we discuss several possible applications of the majority C-coloring, pose a conjecture and several open problems related to $\mc$-edge-stability and  $\mc$-edge-criticality of graphs.

\subsection{Motivation} \label{subsec:motivation}
We shortly introduce the hypergraph and graph coloring models from the literature that inspired the introduction of the majority C-coloring.

\paragraph{C-coloring of hypergraphs.} A hypergraph $\cH=(X,\cE)$ is a set system  $\cE$ over the underlying vertex set $X$. More precisely,  $\cE \subseteq 2^X \setminus \{\emptyset\}$ is required. However, in the context of C-coloring, one-element subsets of $X$ are also excluded from the hyperedge set $\cE$. If each $e \in \cE$ contains exactly two vertices, then $\cH$ is $2$-uniform and corresponds to a (simple) graph. 

A vertex coloring $\phi \colon  X \rightarrow \mathbb{N}$ is a C-coloring of the hypergraph $\cH=(X, \cE)$, if every hyperedge $e \in \cE$ contains at least two vertices assigned the same color. The largest number of different colors that can be used in a C-coloring of $\cH$ is the upper chromatic number $\overline{\chi}(\cH)$ of the hypergraph. 

Although sporadic related studies were conducted earlier (see e.g.~\cite{Arocha, Sterboul}), systematic research on C-coloring of hypergraphs was not initiated until the seminal papers of Voloshin~\cite{Vol-1993, Vol-1995}. The extensive study of this topic revealed several unexpected phenomena when the C-coloring constraint is applied together with the classical constraint that prohibits the presence of monochromatic hyperedges~\cite{Bacso, Blazsik, Bujtas-2009, Jiang-2002, Kundgen}. For more details and results, we refer the reader to the book~\cite{Vol-2002}, the survey paper~\cite{BuTu-2010}, and the references therein.

A trivial but important fact is that the C-coloring of a graph (i.e., a $2$-uniform hypergraph) $G$, makes every component monochromatic, and therefore the upper chromatic number equals the number of components of $G$. 

\paragraph{$k$-improper C-coloring of graphs.} 
In light of some motivating problems, the following relaxation of the C-coloring constraint was introduced and studied for graphs in~\cite{BuSaTuPuVa-10}. If $k$ is a positive integer and $G$ is a graph, then  $\phi \colon  V(G) \rightarrow \mathbb{N}$ is a $k$-{\it improper C-coloring} if every vertex $v \in V(G)$ has at most $k$ neighbors which receive colors different from $\phi(v)$.  The largest number of colors that can be used in a $k$-improper C-coloring of $G$ is the $k$-improper upper chromatic number $\overline{\chi}_{k\mbox{\footnotesize{-imp}}}(G)$. 

The definitions directly imply that for $r$-regular graphs, $\lfloor \frac{r}{2} \rfloor$-improper C-colorings correspond to majority C-colorings.

\paragraph{Satisfactory partitions.} 
%The motivational example given in \cite{GeKob-2000} for the satisfactory graph partitioning problem describes a partition of a group of conference participants into two boats for sightseeing, such that everyone knows at least as many persons in his/her boat as in the other one, and neither boat is empty. 
Given a graph $G$ and a (nontrivial) partition $\cP=\{V_1, V_2\}$ of its vertex set $V$, we say that $\cP$ is a {\it satisfactory partition} if each vertex $v \in V_i$ has at least as many neighbors in $V_i$ as in $V_{3-i}$, for $i \in \{1,2\}$. It is straightforward to see that $G$ is (satisfactorily) partitionable if and only if $\mc(G) \ge 2$ that is $G$ admits a majority C-coloring with two colors.

The results in \cite{BTV-2003, GeKob-2000, Moshi, Shaf} show that all cubic graphs except $K_4$ and $K_{3,3}$ are
partitionable in linear time, and all $4$-regular graphs except $K_5$ are partitionable in linear time. Therefore, $\mc(G) \ge 2$ holds for every $3$- and $4$-regular graph, except for $K_4$, $K_{3,3}$, and $K_5$. Further, as proved in~\cite{GeKob-2004}, the same inequality $\mc(G) \ge 2$ holds if $G$ is a graph different from a star with a girth of at least $5$. For further results, we refer to the survey paper~\cite{BTV-06}.

A more general variant, called the  {\it sum satisfactory $k$-partition} was introduced~\cite{BTV-06, GeKob-1998} as a (nontrivial) partition  $\cP=\{V_1,\dots, V_k\}$ of $V$, where each vertex $v \in V_i$ has at least half of its neighbors in the same class $V_i$, for $i \in [k]$. However, the study of this problem was brief and focused solely on algorithmic questions. For this reason, although a $\cg$-coloring with $k$ colors corresponds to a sum satisfactory $k$-partition, we decided to use the name majority C-coloring to connect our topic to the C-coloring problems.

\paragraph{Majority coloring.}
A coloring of the vertices of a graph is called a \textit{majority coloring} if for every vertex $v$ at most half of its neighbours receive the same color as $v$.
It follows from the more general result of Lovász~\cite{Lovasz} that every undirected graph admits a majority coloring with two colors. The corresponding problem for directed graphs has been studied in~\cite{CaXiYa-2025, KeOuSeZyWo-17}.

Both majority coloring and majority C-coloring are based on imposing local constraints on monochromatic neighborhoods: for each vertex, the number of neighbors sharing its color is bounded by a prescribed threshold. However, the two notions differ substantially in their optimization goals. Nevertheless, there exist graph classes for which the majority chromatic number coincides with the majority C-chromatic number; an example is the graph $K_{4,4}$ shown in Fig.~\ref{fig:ex_substar_cubic}(c). 

We also note that an edge analogue of majority coloring has been considered in~\cite{BoKaPaPlRaWo-2023}.

%  We will use both names interchangeably in our paper.

%\begin{definition}[\textcolor{orange}{Upper?} $\cg$-chromatic number]\label{def:mc}
%\emph{For a graph $G$ define the} upper $\cg$-chromatic number \emph{of $G$, denoted by $\mc(G)$, as the largest integer $k$ for which $G$ admits a $\cg$-coloring:}
%\[\mc(G)\;:=\;\max\bigl\{\,k\in\mathbb{N}:\ \exists\ f:V\to [k] \ \text{\emph{which is a $\cg$-coloring of }}G\,\bigr\}.\] \end{definition}

%%%%%%%%%%%%%%%%%%%%%%%%%%%%%%%%%%%%%%%%%%%%%%%%%
\section{Preliminary results } 
\label{sec:pre}
%%%%%%%%%%%%%%%%%%%%%%%%%%%%%%%%%%%%%%%%%%%%%%%%
\begin{definition}[Majority C-coloring] %($\cg$-coloring)]
\label{def:cg}
Let $G = (V, E)$ be an undirected graph and $k$ a positive integer.  
A \textit{majority C-coloring} (\textit{$C_{\geqslant}$-coloring}, for short) of $G$ with $k$ colors is a surjective function $ \phi : V \rightarrow [k] $
such that every vertex $v \in V$ satisfies
\begin{equation} \label{eq:majority-1}
    \bigl|\{\, u \in N(v) : \phi(u) = \phi(v) \,\}\bigr| \ge \frac{\deg_G(v)}{2}.
\end{equation} 
The largest integer $k$ for which $G$ admits a $\cg$-coloring is called \emph{majority C-chromatic number} (\textit{$C_{\geqslant}$-chromatic number}, for short) and it is denoted by $\mc(G)$. A majority C-coloring of a graph $G$ with $\mc(G)$ colors is called a $\mc$-{\it coloring} of $G$. 
\end{definition}
For a majority C-coloring $\phi$ of $G=(V,E)$ that uses $k$ colors, we define the color classes $V_i=\phi^{-1}(i)$, for $i \in [k]$, and use this notation throughout the paper when $\phi$ is clear from the context. If a vertex $v$ belongs to the color class $V_i$, condition~\eqref{eq:majority-1} can be formulated equivalently as
\begin{equation} \label{eq:majority-2}
    |N(v) \cap V_i| \ge |N(v) \cap (V\setminus V_i)|.
\end{equation}

Definition~\ref{def:cg} directly implies the following basic properties.
\begin{observation} \label{obs:basic} Let $\phi$ be a majority C-coloring of the graph $G$.
\begin{enumerate}[label=\normalfont(\roman*)]
    \item If $v$ is a leaf in $G$ and $vu \in E(G)$, then $\phi(v)=\phi(u)$.  
    \item If $\phi$ uses $k$ colors and $k \ge 2$, then there exists a majority C-coloring with $(k-1)$-colors. In particular, for every $k'$ with $1 \leq k' \leq \mc(G)$, the graph has a majority C-coloring with $k'$ colors.
    \item If $\phi$ uses $\mc(G)$ colors, then every color class induces a connected subgraph in $G$.
    \item If  $G$ is a disconnected graph with components $G_1, \dots, G_k$, then $\mc(G)= \sum_{i=1}^k \mc(G_i)$.
\end{enumerate}   
\end{observation}
Statement (i) immediately follows from~\eqref{eq:majority-1}. Concerning part (ii), we note that by merging two color classes in a $\cg$-coloring, the majority condition~\eqref{eq:majority-1} remains satisfied. Concerning (iii), suppose that a color class induces a subgraph with components $H_1, \dots , H_{\ell}$, where $\ell \ge 2$. By recoloring the vertices of $H_1$ with a new color, we obtain a $\cg$-coloring of $G$ with $\mc(G)+1$ colors that is impossible by definition. Statement (iii) and Definition~\ref{def:cg} directly imply (iv).

\vspace{5mm} %5mm vertical space

The above observations, together with simple structural arguments, allow us to determine exact values of $\mc(G)$ for some basic graph classes. %We briefly indicate the reasoning behind these values below.

\begin{observation} \label{obs:simple_graph_classes}
For the following graph classes, the majority C-chromatic number satisfies:
\begin{enumerate}[label=\normalfont(\roman*)]
\item For a complete graph $K_n$, $n\geq 1$, $\mc(K_n)=1$.
    
    \item For a cycle $C_n$ and a path $P_n$, $n\geq 3$, $\mc(C_n)=\mc(P_n)=\lfloor \frac{n}{2}\rfloor$.
    %\item For a path $P_n$, $n\geq 2$, $\mc(P_n)=\lfloor n/2\rfloor$.
\item For a star $S_n$, $n\geq 0$, on $n+1$ vertices, $\mc(S_n)=1$; and for a subdivided star $S(S_n)$, we have $\mc(S(S_n))=\lfloor \frac{n}{2}\rfloor + 1$.
   % \item For a subdivided star $S(S_n)$, $\mc(S(S_n))=\lfloor n/2\rfloor + 1$.
    \item For a wheel $W_n$, $n\geq 3$, on $n+1$ vertices, $\mc(W_n)=1$.
    \item For a complete bipartite graph $K_{m,n}$, $m,n\geq 1$, $\mc(K_{m,n})=2$ if both $m,n$ are even and $\mc(K_{m,n})=1$ otherwise.
\end{enumerate}
\end{observation}
\begin{proof}
\noindent (i) Note that every vertex of $K_n$ is of degree $n-1$, so if we color a vertex $v$ with any color, the relevant color class in any $\cg$-coloring of $K_n$ needs to have at least $1+\frac{1}{2}(n-1)=\frac{1}{2}(n+1)$ vertices. Thus, it is impossible that $\mc(K_n)>1$.

\noindent (ii) It is clear that every color class in any $\cg$-coloring of $P_n$ or $C_n$ contains at least two vertices. Hence, the $\cg$-chromatic number is at most $\lfloor \frac{n}{2}\rfloor$, and a $\cg$-coloring with this number of colors is easy to obtain. Note that the formula of $\mc(C_n)$ corresponds to $\overline{\chi}_{1\mbox{\footnotesize{-imp}}}(C_n)$. The latter was already established in~\cite{BuSaTuPuVa-10}.

%\hf{Justification for $C_n$ to integrate with the rest:} Since cycles are $2$-regular, majority $C$-colorings of $C_n$ are equivalent to $1$-improper $C$-colorings, which implies the result (cf. \cite{BuSaTuPuVa-10}).

%\hf{Justification for $P_n$ to integrate with the rest:}The case of paths can be treated analogously as for cycles. Indeed, the path $P_n$ is obtained from the cycle $C_n$ by deleting a single edge. Since this operation does not affect the local degree conditions relevant to majority $C$-colorings, we receive the result for $n\geq 3$. For shorter paths it can be checked directly.

\noindent (iii) The value for a star follows directly from Observation~\ref{obs:basic}(i).  For a subdivided star, consider the central vertex $r$ of $\deg(r) = n$. At least $\lceil \frac{n}{2} \rceil$ of the adjacent vertices must receive the color of $r$. Hence, at most $\lfloor \frac{n}{2} \rfloor$ neighbors of $r$ can be colored differently, which, together with Observation~\ref{obs:basic}(i), yields $\mc(S(S_n))=\lfloor \frac{n}{2} \rfloor+1.$
An example of such a coloring is presented in Fig.~\ref{fig:ex_substar_cubic}(a).

\noindent (iv) Suppose that a majority $\cg$-coloring $\phi$ of $W_n$ uses more than one color and let $r$ be the central vertex. Consider a vertex $v$ from the cycle such that $\phi(v) \neq \phi(r)$. 
The degree constraints then force all vertices from the cycle to receive the same color $\phi(v)$, distinct from $\phi(r)$, which violates the majority condition at $r$. Hence $\mc(W_n)=1$.

\noindent (v) Observe that each color class in any $\cg$-coloring of $K_{m,n}$ needs to contain at least $\lceil\frac{m}{2}\rceil$ vertices from one part of the graph and at least $\lceil\frac{n}{2}\rceil$ vertices from the other part. This simple fact forces two colors with $m,n$ even, and one color if any of these numbers is odd. See Fig.\ \ref{fig:ex_substar_cubic}(b) and (c) for examples.   \end{proof}

%The argument for wheels, stated in item (iii), is slightly more complex. Assume $\cg$-coloring is possible with more than 1 color. Assign to $r$ a color $\phi(r)$. To satisfy majority condition at least half of the cycle vertices must also receive $\phi(r)$. Every vertex of the outer cycle has degree three. Hence, every such vertex must have at most one neighbor of different color. Since $r$ is already colored with $\phi(r)$, every cycle vertex must share the same color, distinct from $\phi(r)$. This, however, violates majority condition for $r$. Hence, such coloring is impossible and we conclude that $\mc(W_n)= 1$.

   % Observe that $S'_2 \simeq P_5$ and its  $\cg$ chromatic number can be calculated both from the formula (\ref{eq:path}) and from Theorem \ref{thm:subdstar}.

  At the end of this part, we recall that the decision problem of whether a graph $G$ admits a sum satisfactory $k$-partition was proved to be NP-complete~\cite{BTV-06} for every fixed $k\ge 2$. As noted in Section~\ref{subsec:motivation}, this question is equivalent to deciding whether $\mc(G) \ge k$ holds. Therefore, we may conclude the following key property.

 \begin{proposition} \label{prop:NP-c}
     For every fixed integer $k \ge 2$, it is NP-complete to decide whether $\mc(G) \ge k$ holds for a general graph $G$.
 \end{proposition}
 
\subsection{General upper bounds}
\begin{proposition} 
    Let $G$ be a graph on $n$ vertices with minimum degree $\delta(G)$ and maximum degree $\Delta(G)$. Then $$\mc(G) \leq 1+\frac{n - \Bigl\lceil \frac{\Delta(G)}{2}\Bigr\rceil -1}{\Bigl\lceil \frac{\delta(G)}{2}\Bigr\rceil +1}.$$
    \label{prop:upperboundDelta}
\end{proposition}
\begin{proof}
Let $\mc(G)=k$ %and $\phi$ be a $\mc$-coloring of $G$. %$\phi:V\to [k]$ be a $\cg$-coloring of $G$ such that $\mc(G)=k$, 
  and let  $V_1,V_2,\ldots ,V_k$ be the color classes in a $\mc$-coloring of $G$. Let $v$ be a vertex of maximum degree in $G$. We may assume, without loss of generality, that  $v \in V_1$. Thus, by the majority condition~\eqref{eq:majority-1},
$$ |V_1| \ge |N(v) \cap V_1| +1 \ge \Bigl \lceil \frac{\Delta(G)}{2}\Bigr \rceil +1.
$$
Similarly, each remaining color class $V_i$, for $2 \le i \le k$, contains a vertex of degree at least $\delta(G)$, and its size can be estimated as  $|V_i| \ge \bigl\lceil \frac{\delta(G)}{2} \bigr\rceil +1$. We therefore infer 
$$(k-1)\Bigl(\Bigl \lceil \frac{\delta(G)}{2}\Bigr \rceil +1 \Bigr) \leq     n - \Bigl\lceil \frac{\Delta(G)}{2}\Bigr\rceil -1,$$ 
 and the desired inequality follows.  
\end{proof}

\begin{corollary}
     Let $G$ be a graph of order $n$. Then $$\mc(G)\le \frac{n}{\Bigl\lceil \frac{\delta(G)}{2}\Bigr\rceil +1}.$$ 
    \label{col:upperbounddelta}
\end{corollary}
Several results of this paper demonstrate that the upper bounds in Proposition~\ref{prop:upperboundDelta} and Corollary~\ref{col:upperbounddelta} are sharp. Proposition~\ref{prop:cycle-k} will show that these bounds are tight for $C_n^k$ whenever $n$ is divisible by $k+1$ and $n \ge 2k+2$. Proposition~\ref{prop:cubic} will provide another infinite class of sharp examples by implying that every $(\text{claw}, \text{diamond}, K_4)$-free cubic graph $G$ of order $n$ satisfies $\mc(G) = \frac{1}{3}n$. 

%Note that the bound from Corollary~\ref{col:upperbounddelta} is sharp, for example for all even cycles and also for (claw, diamond)-free cubic graphs. While the first class of even cycles is easy to verify, in the latter case first we note that for each vertex in a (claw, diamond)-free cubic graph, the neighborhood $N(v)$ induces either $K_3$ or $K_2\cup K_1$. Consequently, every connected such graph is either $K_4$ or a ring of even number of triangles. Note now that the vertex set of a (claw, diamond)-free cubic graph $G$, excluding $K_4$, can be partitioned into triangles. 
%If every triangle gets its private color, we get a $\cg$-coloring of $G$ that shows $\mc(G) \ge n/3$. Therefore,  $\mc(G)=n/3$. As $\delta(G)=3$, the inequality in the corollary holds with equality for $G$.
%One can easily check that (claw, diamond)-free cubic graphs, excluding $K_4$, are also good example for the  tightness of the bound from Proposition~\ref{prop:upperboundDelta}
\subsection{Edge deletion and spanning subgraphs}

In this section, we start with the investigation of how the $\cg$-chromatic number of a graph behaves under the deletion of edges.
%In Section \ref{sec:graph_classes}, we already considered how removing any edge in a complete graph affects the $\cg$-chromatic number. It turned out that removing one edge from a complete graph does not actually change the value of this parameter. Now we consider the effect of the edge removal operation on any graph. We will examine how much the $\cg$-chromatic number can change after such an operation. 
It turns out that if a subgraph is obtained by the deletion of a single edge,  the $\cg$-chromatic number might still increase or decrease. However, for this case, we can prove sharp lower and upper bounds.

\begin{proposition}\label{prop:G-e}
    For every graph $G$ and every edge $e \in E(G)$, it holds that
    \begin{equation}
        \mc(G)-2 \leq \mc(G-e) \leq \mc(G) +1.\label{ineq:del_edge}
    \end{equation} 
Moreover, both bounds are sharp.
\end{proposition}

\begin{proof}
    Let $e=v_1v_2$ be an edge in $G$ and $\phi'$ a $\mc$-coloring of $G'=G-e$. Then we add the edge $e$ to $G'$, and consider the assignment $\phi'$ in $G$. The majority C-coloring condition remains satisfied for every vertex $u$ different from $v_1$ and $v_2$. However, since $\deg_G(v_i)=\deg_{G'}(v_i)+1$, for $i \in \{1,2\}$, it is possible that $\lceil\frac{1}{2}\deg_G(v_i)\rceil = \lceil\frac{1}{2} \deg_{G'}(v_i)\rceil +1$. If $\phi'(v_1) \neq \phi'(v_2)$, it might result in a situation when $v_1$ or $v_2$ has one less neighbor with a common color than required in $G$. In this case, we define a coloring $\phi$ by merging the color classes of $v_1$ and $v_2$. This coloring of $G$ uses $\mc(G-e)-1$ colors, and it is a $\cg$-coloring of $G$. Consequently, $\mc(G-e) -1 \leq \mc(G)$.

        To prove the lower bound, we take a $\cg$-coloring $\varphi$ of $G$ which uses $\mc(G)$ colors. Let $e=v_1v_2$ be an edge of $G$. If $\varphi$ is not a $\cg$-coloring for $G'=G-e$,  then $\varphi(v_1)=\varphi(v_2)$ and $v_i$ has only 
        \begin{equation} \label{eq:G-e}
           \Bigl\lceil \frac{\deg_G(v_i)}{2} \Bigr\rceil -1 < \Bigl\lceil \frac{\deg_{G'}(v_i)}{2} \Bigr\rceil= \Bigl\lceil \frac{\deg_{G}(v_i)-1}{2} \Bigr\rceil 
        \end{equation}
       neighbors with the same color in $G'$, for $i=1$, or $i=2$, or for both. If inequality~\eqref{eq:G-e} holds, then $\deg_G(v_i)$ is even and $v_i$ has a neighbor $u_i$ with $\varphi(u_i) \neq \varphi(v_i)$. By merging the color classes of $u_i$ and $v_i$, we obtain a coloring that satisfies the majority condition for $v_i$. If~\eqref{eq:G-e} holds for both $v_1$ and $v_2$, then we might have to merge two further color classes. This way, we obtain the vertex coloring $\varphi'$. As $\varphi'$ is a $\cg$-coloring of $G-e$, and it uses at least $\mc(G)-2$ colors, we conclude $\mc(G)-2 \leq \mc(G-e)$.

To show the tightness of the lower bound, we construct a family $\mathcal{H}$ of graphs $H$ as follows.
Starting from a cycle $C_n$ with vertex sequence
$v_1,v_2,\ldots,v_n,v_1$, we attach a path $P_5$ to each vertex $v_i$ by identifying its middle vertex with $v_i$, for every $i \in [n]$.
Consequently, the graph $H$ has $5n$ vertices and $5n$ edges. It can be readily verified that $\mc(H)=2n+1$ and that $\mc(H - v_1v_2)=2(n-1)+1=2n-1$.
\\For the upper bound, we construct a family $\mathcal{F}$ of graphs $F$ obtained as follows.
Take two odd cycles $C_{2k+1}$, with vertex sequences
$u_1,\dots,u_{2k+1},u_1$ and $v_1,\dots,v_{2k+1},v_1$,
respectively, and add the edges $u_1v_1$ and $u_kv_k$.
It is straightforward to verify that $\mc(F)=2k-1$,
whereas $\mc(F - u_1v_1)=2k$.
\end{proof}

The sharpness of the two bounds in~\eqref{ineq:del_edge} implies that there is no general monotonicity relation between the parameter $\mc(G)$ of a graph $G$ and that of its spanning subgraphs. The following examples show further evidence for this fact.
Let $m\ge 2$. Recall that the Dutch windmill graph $D_{3,m}$, also
known as the friendship graph,  is formed by taking $m$ copies of the cycle $C_3$ and identifying one vertex from each copy into a single common vertex.
%is obtained by taking $m$ copies of the cycle $C_3$ sharing a common vertex. 
It is easy to observe that $\mc(D_{3,m})=\left\lfloor \frac{m}{2} \right\rfloor +1.$
Observe that the star $K_{1,2m}$ is a spanning subgraph of
$D_{3,m}$, and hence there exists a spanning subgraph $H$ of a graph
$G$ such that
$\mc(H)=\mc(K_{1,2m})=1$
 while $\mc(G)=\mc(D_{3,m})\ge 2$. On the other hand, for $n\ge 4$, the cycle $C_n$ is a spanning
subgraph of the complete graph $K_n$. In this case we have
$\mc(G)=\mc(K_n)=1$
and $
\mc(H)=\mc(C_n)= \left\lfloor \frac{n}{2} \right\rfloor \ge 2$, showing that the value of $\cg$-chromatic number may also increase when passing to a spanning subgraph. 

Altogether, these examples further demonstrate that the majority C-chromatic number is neither monotone increasing nor monotone decreasing with respect
to taking spanning subgraphs.

\section{Some graph classes}\label{sec:graph_classes}
    
\subsection{Powers of paths and cycles}
Recall that $P_n^k$, the $k$-th power of a path of order $n$, is obtained from $P_n$ by making adjacent every pair of vertices that are within distance $k$ in $P_n$. For $n \leq k+1$, the diameter of $P_n$ is at most $k$, and hence $P_n^k \cong K_n$, and by Observation~\ref{obs:simple_graph_classes}(i), it follows that $\mc(P_n^k)=1$. Therefore, we restrict our attention to the case $n \ge k+2$.

\begin{theorem} \label{thm:power-paths}
For every two positive integers $n$ and $k$ with $n \ge k+2$, it holds that
    \begin{align} \label{eq:thm-path-k}
    \mc(P_n^k)= \Bigl\lfloor\frac{n}{k+1}\Bigr\rfloor.
\end{align}
\end{theorem}
\begin{proof}
    Let $G=P_n^k$ be the $k$-th power of the path $P_n \colon v_1, \dots, v_n$ and define $I[a,b]=\{v_a, \dots, v_b\}$, for every two integers $a$, $b$ with  $1\le a \le b \le n$. We first note that the coloring that assigns color $s$ to the vertices in $I[(s-1)(k+1)+1, s(k+1)]$, for every $1 \le s \le \lfloor \frac{n}{k+1} \rfloor -1$, and assigns the color $q= \lfloor \frac{n}{k+1} \rfloor $ to the remaining vertices in $I[(q-1)(k+1)+1, n]$ satisfies the majority condition~\eqref{eq:majority-1} (cf. Fig.~\ref{fig:P8-cube-arcs}). Consequently, $\mc(P_n^k) \ge \lfloor \frac{n}{k+1} \rfloor$ holds for every $n$ and $k$. In the continuation, we prove that the reverse inequality also holds if $n \ge k+2$.
    \medskip
    
    Let $\mc(G)=\ell$ and let $\phi$ be a $\mc$-coloring of $G$. The color classes associated with $\phi$ will be denoted by $V_1, \dots, V_\ell$. 
    We prove the following key property.
    \begin{claim} \label{claim:0}
        If $1 \le i \le k$ and $i \leq n-k-1$, then $\phi(v_i)=\phi(v_{i+1})$.
    \end{claim} 
    \proof Suppose, to the contrary, that $v_i \in V_p$, $v_{i+1} \in V_q$, and $p \neq q$. The conditions $i \leq k$ and $i \leq n-(k+1)$ imply that $N[v_i]=I[1,i+k]$ and $N[v_{i+1}]=I[1, i+k+1]$. Applying the majority condition~\eqref{eq:majority-1} to $v_i$ and $v_{i+1}$, respectively, we obtain
    \begin{align*}
        |V_p\cap I[1,i+k+1]| &\ge \Bigl\lceil \frac{\deg_G(v_i)}{2} \Bigr\rceil +1= \Bigl\lceil \frac{i+k-1}{2} \Bigr\rceil +1.\\
        |V_q\cap I[1,i+k+1]| &\ge \Bigl\lceil \frac{\deg_G(v_{i+1})}{2} \Bigr\rceil +1= \Bigl\lceil \frac{i+k}{2} \Bigr\rceil +1.   
        \end{align*}
        Since $V_p \cap V_q = \emptyset$, these imply 
        $$|I[1, i+k+1]|=i+k+1 \ge \Bigl\lceil \frac{i+k-1}{2} \Bigr\rceil + \Bigl\lceil \frac{i+k}{2} \Bigr\rceil +2= i+k+2.
        $$
       This contradiction proves that $\phi(v_i) \neq \phi(v_{i+1})$ is impossible under the given conditions. \smallqed
   \medskip
   
   In the continuation, we distinguish two cases.
  \paragraph{Case 1.} $2k+2 \le n$.\\
    We prove that every color class $V_j$ contains at least $k+1$ vertices. The condition $2k+2 \le n$ ensures $k < n-k-1$, and then Claim~\ref{claim:0} implies that $I[1,k+1]$ is monochromatic. Consequently, if $V_j$ intersects $I[1,k+1]$, then $|V_j| \ge k+1$. By symmetry, the same is true if $V_j$ intersects $I[n-k,n]$. If $V_j \subseteq I[k+2, n-k-1]$, then $V_j$ contains a vertex of degree $2k$ and, by~\eqref{eq:majority-1}, we may derive $|V_j| \ge k+1$. The conclusion is that every color class contains at least $k+1$ vertices and $\mc(G) \leq \lfloor \frac{n}{k+1} \rfloor$.
     \paragraph{Case 2.} $k+2 \le n \le 2k+1$.\\
     Let $B_1=I[1,n-k]$ and $B_2=I[k+1,n]$. If $n \leq 2k-1$, then there are vertices outside $B_1 \cup B_2$ and we define also $A=I[n-k+1, k]$ (see Fig.~\ref{fig:pow_path}).  Under the present conditions $n-k-1 \le k$ and then Claim~\ref{claim:0} shows that $B_1$ is monochromatic. By symmetry, $B_2$ is monochromatic as well. Consequently, if a color class $V_j$ intersects $B_1$, then $v_{n-k} \in V_j$. Since $\deg(v_{n-k})=n-1$, the majority condition~\eqref{eq:majority-1} gives
     \begin{equation} \label{eq:pow-path}
         |V_j| \ge \Bigl\lceil\frac{n-1}{2} \Bigr\rceil +1\ge \frac{n+1}{2}.
     \end{equation}
     If $V_j$ meets $B_2$, then $v_{k+1} \in V_j$ and $\deg(v_{k+1})=n-1$ implies the same lower bound~\eqref{eq:pow-path}. If $V_j$ contains a vertex from the remaining part $A$, this vertex is of degree $n-1$, and we get $|V_j| \ge \frac{1}{2}(n+1)$ again. Consequently, every color class contains at least $\frac{1}{2}(n+1)$ vertices and thus, $\phi$ uses only one color. This verifies $\mc(G)= \lfloor \frac{n}{k+1} \rfloor$ for the second case, and the proof is complete. 
\end{proof}

\begin{figure}[htb]
\centering
\begin{tikzpicture}[
scale=1,
vertex/.style={circle, draw, fill=white, inner sep=2pt},
lab/.style={font=\scriptsize},
group/.style={draw, rectangle, rounded corners, inner sep=6pt}
]

% --- B1 vertices (4 vertices) ---
\node[vertex,label={[lab,yshift=-2mm]below:$1$}] (v1) at (0,0) {};
\node at (1,0) [fill=white,inner sep=1pt] {$\dots$};
%\node[vertex] (v2) at (1,0) {};
%\node[vertex,label={[lab,xshift= -1.5mm, yshift=-2mm]below:$n-2k+1$}] (v3) at (2,0) {};
%\node at (2.5,0) [fill=white,inner sep=1pt] {$\dots$};
\node[vertex,label={[lab,xshift= -1mm, yshift=-2mm]below:$n-k$}] (v4) at (2,0) {};

% --- A vertices ---
\node[vertex,label={[lab,xshift= 2mm,yshift=-2mm]below:$n-k+1$}] (v5) at (3,0) {};
\node at (4,0) [fill=white,inner sep=1pt] {$\dots$};
\node[vertex,label={[lab, xshift= 0.3mm, yshift=-2mm]below:$k$}] (v6) at (5,0) {};

% --- B2 vertices ---
\node[vertex,label={[lab, xshift= 0.5mm, yshift=-2mm]below:$k+1$}] (v7) at (6,0) {};
\node at (7,0) [fill=white,inner sep=1pt] {$\dots$};
\node[vertex,label={[lab,yshift=-2.5mm]below:$n$}] (v8) at (8,0) {};

% edges (background)
\begin{scope}[on background layer]
\draw[thick] (v1) -- (v4) -- (v5) -- (v6) -- (v7) -- (v8);
\end{scope}

% --- Top groups ---
\node[group, label=above:$B_1$, fit=(v1)(v4)] {};
\node[group, label=above:$A$,   fit=(v5)(v6)] {};
\node[group, label=above:$B_2$,  fit=(v7)(v8)] {};

% --- Bottom subgroups of B1 ---
%\node[group, label={[yshift=1mm]above:$B_{1,1}$}, fit=(v1)(v2)] {};
%\node[group, label={[yshift=1mm]above:$B_{1,2}$}, fit=(v3)(v4)] {};

\end{tikzpicture}
\caption{A schematic drawing of the partition used in the proof of Theorem~\ref{thm:power-paths} for Case 2 with $k+2\leq n\leq 2k-1$.} \label{fig:pow_path}
\end{figure}
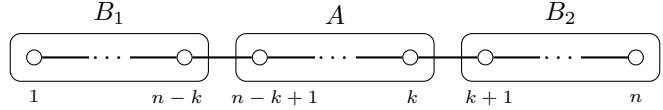

Since $P_n^{n-2} \cong K_n-e$ for any $n\geq 3$, and since $K_2-e$ is the edgeless graph on two vertices, we immediately have the following. 

\begin{corollary}\label{cor:Kn-e}
    If $e$ is an edge in the complete graph $K_n$, for $n \ge 2$, then 
\[
\mc(K_n-e)=
\begin{cases}
1, & \text{for } n\ge 3,\\[4pt]
2, & \text{for } n=2.
\end{cases}
\]
\end{corollary}

\begin{figure}[htb]
\centering
\begin{tikzpicture}[
    scale=1.0,
    every node/.style={circle, draw=black, fill=white, inner sep=1.5pt}
]
% Wierzchołki P_8
\foreach \i/\x/\col in {
  0/0/1, 1/1/1, 2/2/1, 3/3/1,
  4/4/2, 5/5/2, 6/6/2, 7/7/2
} {
  \node (v\i) at (\x,0) {\col};
}

% Krawędzie P_8 (odległość 1) – proste
\foreach \i in {0,...,6} {
  \pgfmathtruncatemacro{\j}{\i+1}
  \draw[thick] (v\i) -- (v\j);
}

% Krawędzie odległości 2 – niskie łuki
\foreach \i in {0,...,5} {
  \pgfmathtruncatemacro{\j}{\i+2}
  \draw[thick] (v\i) to[bend left=25] (v\j);
}

% Krawędzie odległości 3 – wyższe łuki
\foreach \i in {0,...,4} {
  \pgfmathtruncatemacro{\j}{\i+3}
  \draw[thick] (v\i) to[bend left=45] (v\j);
}

\end{tikzpicture}
\caption{$\mc$-coloring of $P_8^3$.}
\label{fig:P8-cube-arcs}
\end{figure}
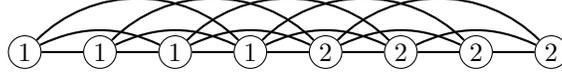

For powers of cycles, the situation is much simpler, and we can easily establish the following statement.
\begin{proposition} \label{prop:cycle-k}
    For every two positive integers $n$ and $k$ with $n\geq 3$ and $n\geq 2k+1$, it holds that 
    \begin{align*} 
    \mc(C_n^k)=\Bigl\lfloor\frac{n}{k+1}\Bigr\rfloor.
\end{align*}
\end{proposition}
\begin{proof} 
If $2k+1 \le n$, the graph $C_n^k$ is $(2k)$-regular and therefore, Corollary~\ref{col:upperbounddelta} directly implies that $\mc(C_n^k)\leq \lfloor\frac{n}{k+1}\rfloor$, while a $\cg$-coloring with $\lfloor\frac{n}{k+1}\rfloor$ colors can be obtained as described for $P_n^k$ in the proof of Theorem~\ref{thm:power-paths}.
\end{proof}

\begin{comment}
    \begin{figure}[htb!]
\centering
\begin{tikzpicture}[scale=1.0, every node/.style={circle, draw=black, fill=white, inner sep=1.5pt}]

% Wierzchołki głównej ścieżki (10)
\foreach \i/\x/\col in {
  0/0/1, 1/1/1, 2/2/2, 3/3/2, 4/4/3,
  5/5/3, 6/6/4, 7/7/4
} {
  \node (v\i) at (\x,0) {\col};
}

% Krawędzie ścieżki głównej
\foreach \i in {0,...,6} {
  \pgfmathtruncatemacro{\j}{\i + 1}
  \draw[thick] (v\i) -- (v\j);
}
\end{tikzpicture}
\caption{$\cg$-coloring of $P_{8}$ with $\mc(P_{8})$ colors. }\label{fig:P10}
\end{figure}
\end{comment}

Note that the powers of paths and the powers of cycles are examples of graphs with arbitrary high value of $\cg$-chromatic number.

%Let $m\ge 2$. Recall that the Dutch windmill graph, denoted by $D_3^{(m)}$, also called a friendship graph, is the graph obtained by taking $m$ copies of the cycle $C_3$ with a vertex in common. Observe that $\mc(D_3^{(m)})=\lfloor m/2 \rfloor +1$. 

%Let $H$ be a spanning subgraph of $G$. We claim that $\mc(G)$ and $\mc(H)$ are incomparable, which is indeed the case since a star $K_{1,2m}$ is a spanning subgraph of a Dutch wind mill graph $D_3^{(m)}$, $m\ge 2$, and $\mc(H)=\mc(K_{1,2m})=1<2\le \mc(D_3^{(m)})=\mc(G)$. And on the other side, a cycle $C_n$, $n\ge 3$, is a spanning subgraph of a complete graph $K_n$ and in this case for $n\ge 4$ we have $\mc(G)=\mc(K_n)=1<2\le \mc(C_n)=\mc (H)$.

\subsection{Cubic graphs} First, consider the Petersen graph $P$. Based on the results obtained for 1-improper C-coloring for this graph in \cite{BuSaTuPuVa-10}, we conclude that $\mc(P)=2$. By the results obtained in~\cite{BTV-06} that every cubic graph except $K_4$ and $K_{3,3}$ is partitionable, we know that every cubic graph except $K_4$ and $K_{3,3}$ admits a majority C-coloring with at least two colors. Our goal in this section is to prove a sharp upper bound on the majority C-chromatic numbers of  cubic graphs. 
%We also show that this upper bound cannot be improved in general: cubic graphs form a class of graphs whose majority c-chromatic number is unbounded. More precisely, we construct cubic graphs whose $\cg$-chromatic number is arbitrarily large. 
\begin{proposition} \label{prop:cubic}
    Let $G$ be an $n$-vertex cubic graph containing $d$ diamonds. Then 
    \begin{enumerate}
    [label=\normalfont(\roman*)]
        \item $\displaystyle \mc(G)\leq \frac{n-d}{3};$
        \item $\displaystyle\mc(G)=\frac{n-d}{3}$ if $G$ is $(\text{claw}, K_4)$-free. %and $n \ge 6$. 
    \end{enumerate}
\end{proposition}
\begin{proof}
   (i) Let $G$ be a cubic graph, and let $\phi$ be a $\cg$-coloring  of $G$. By Corollary~\ref{col:upperbounddelta} we obtain $\mc(G)\leq \frac{n}{3}$. Hence, if $d=0$, we are done. Assume now that $d\geq 1$. Observe that the three vertices of a triangle in $G$ always receive the same color in $\phi$. This implies that each vertex of a diamond also receives the same color in $\phi$. Let $p$ be the number of colors assigned to the vertices belonging to diamonds, and let $k=\mc(G)-p$ be the number of remaining colors. Hence, $\mc(G)=p+k\leq d+k$. Since each color class contains at least 3 vertices, $n\geq 4d+3k$ . Combining these two inequalities, we obtain $\mc(G)\leq d+ \frac{1}{3}(n-4d)= \frac{1}{3}(n-d)$. \\
   (ii) First note that diamond subgraphs in a cubic graph must be pairwise vertex-disjoint. Furthermore, if $G$ is a $({\rm claw},K_4)$-free cubic graph then $V(G)$ can be partitioned into vertex sets $V_1, \dots, V_k$ such that each $V_i$ induces a diamond or a triangle. If $G$ contains $d$ diamonds then we have $d$ partition classes of cardinality $4$ and $\frac{1}{3}(n-4d)$ triangles (classes of cardinality $3$). Therefore, $k=d+ \frac{1}{3}(n-4d)= \frac{1}{3}(n-d)$. By assigning color $i$ to the vertices in $V_i$, for every $i \in [k]$, we obtain a $\cg$-coloring with $k$ colors. This proves $\mc(G) \ge \frac{1}{3}(n-d)$.
\end{proof}

As examples of  graphs with arbitrarily large majority C-chromatic number we can consider the family $\mathcal{G}$ of graphs $G$ defined as follows: 
%\begin{itemize}
    %\item [$G^c$] Let $p$ be an even number. Let us take $p$ cycles $C_3$ and add edges between vertices of these cycles in such a way that the resulting graph, say $G^c$, is cubic (cf. Fig.~\ref{fig:four-triangles}).
    %\item [$G^d$] 
    let $d\ge 2$ and  take $d$ copies of a diamond. We construct a graph $G$ by adding $d$ edges between vertices of degree 2 of two different copies of a diamond in such a way that the resulting graph is cubic and connected (cf. Fig.~\ref{fig:three-diamonds}).
%\end{itemize}
 Every graph $G\in \mathcal{G}$ is a $({\rm claw},K_4)$-free graph with $n=4d$ vertices. Hence, by Proposition~\ref{prop:cubic}(ii), we have $\mc(G)=\frac{1}{3}(n-d)=d$.
 
 %If we assign the same color to each vertex in the same copy of a diamond, while each copy of a diamond receives different color, the resulting coloring is a $\cg$-coloring. Hence, $\mc(G^d)\geq r$. Coversly, observe that ach vertex of a diamond has to receive the same color. It impies, $\mc(G^d)= r=\frac{1}{3}(n-r)$. 
 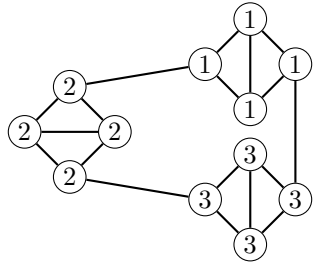
\begin{figure}[htb]
\centering
\begin{tikzpicture}[
    scale=0.6,
    every node/.style={circle, draw=black, fill=white, inner sep=1.5pt}
]
% =====================
% Diamond 1 (lewy)
% =====================
\node (a1) at (-2,  -0.5) {2};
\node (a2) at (-3,  0.5) {2};
\node (a3) at (-2,  1.5) {2};
\node (a4) at (-1, 0.5) {2};
\draw[thick]
(a1)--(a2)--(a3)--(a4)--(a1)
(a2)--(a4); 
% =====================
% Diamond 2 (górny prawy)
% =====================
\node (b1) at ( 1,  2) {1};
\node (b2) at ( 2,  3) {1};
\node (b3) at ( 3,  2) {1};
\node (b4) at ( 2,  1) {1};
\draw[thick]
(b1)--(b2)--(b3)--(b4)--(b1)
(b2)--(b4);
% =====================
% Diamond 3 (dolny prawy)
% =====================
\node (c1) at ( 1, -1) {3};
\node (c2) at ( 2, 0) {3};
\node (c3) at ( 3, -1) {3};
\node (c4) at ( 2, -2) {3};
\draw[thick]
(c1)--(c2)--(c3)--(c4)--(c1)
(c2)--(c4);
% =====================
% Połączenia między diamondami
% =====================
\draw[thick] (a3)--(b1);
\draw[thick] (a1)--(c1);
\draw[thick] (b3)--(c3);
\end{tikzpicture}
\caption{An example of a cubic graph $G\in\mathcal{G}$ and its $\mc$-coloring.}
\label{fig:three-diamonds}
\end{figure}

    \begin{corollary}
For every positive integer $k$, there exists a connected cubic graph $G$ with \mbox{$\mc(G)=k$.}
\end{corollary}    
%%%%%%%%%%%%%%%%%%%%%%%%%%%%%
\section{
%Maximum and minimum size of a graph $G$ with $\mc(G)=k$/ \cs{OR: 
Size, order, $\chi(G)$, and $\mc(G)$}  \label{sec:min-max}

%%%%%%%%%%%%%%%%%%%%%%%%%%%%%%%%%%%%%%%
In this section, different types of connections between the four invariants in the title are discussed. First, the maximum and minimum sizes of a graph are determined in terms of the order and $\cg$-chromatic number. The second subsection shows that, in general, the values of $\chi(G)$ and $\mc(G)$ are incomparable and that their difference can be arbitrarily large. In contrast, it is proved that  $\chi(G) + \mc(G) \leq n +1$ holds for every $n$-vertex graph $G$.

%%%%%%%%%%%%%%%%%%%%%%%%%%%%%%%%%%%%%%%%%%
\subsection{Size, order, and $\mc(G)$} \label{subsec:size-order-mc}
%%%%%%%%%%%%%%%%%%%%%%%%%%%%%%%%%%%%%%%%%%

   Tur\'an's celebrated theorem~\cite{Turan} determines the maximum size of a $K_{k+1}$-free graph $G$ on $n$ vertices. This value (the Tur\'an number $t(n,k)$) equals the maximum number of edges in an $n$-vertex graph $G$ with $\chi(G)=k$. The other extreme is the minimum number of edges when $\chi(G)=k$. If the connectivity of $G$ is not required, this minimum size is clearly ${k \choose 2}$. 

   In this section, we answer the analogous questions for the majority C-coloring. For two given integers $n$ and $k$ with $n \ge k \ge 1$, let $M(n,k)$ and $m(n,k)$ denote the maximum and minimum number of edges, respectively, in an $n$-vertex graph $G$ that has $\mc(G)=k$. Concerning the maximum, Observation~\ref{obs:simple_graph_classes}(i) implies that $M(n,1)= {n \choose 2}$ holds for every positive integer $n$. 
   We proceed by first determining $M(n,2)$ and then extending the result to the higher values of $\mc(G)$. The section is concluded with a simple formula for $m(n,k)$.

   \begin{theorem} \label{thm:M(n,2)}
       For every integer $n \ge 2$, the maximum number of edges in an $n$-vertex graph $G$ with $\mc(G) =2$ is 
       $$
    M(n,2)= 
    \begin{cases}
	\frac{n(n-2)}{2}, &\text{if $n$ is even},\\
	\frac{(n-1)(n-2)}{2}, &\text{if $n$ is odd}.\\
    \end{cases}
    $$
   \end{theorem}
   \begin{proof}
   Let $\overline{M}(n,2)$ denote the minimum number of edges needed to be removed from $K_n$ to obtain a spanning subgraph $G$ with $\mc(G) \ge 2$. As for $\overline{M}(n,2)$ the weaker condition $\mc(G) \ge 2$ is imposed, we may infer $M(n,2) \leq  {n \choose 2}- \overline{M}(n,2)$.
   
   For a given $n \ge 2$, let $x$ be an integer so that $1 \leq x \leq \lfloor \frac{n}{2} \rfloor$.   Consider a graph $G$ on $n$ vertices that has a $\cg$-coloring with two colors such that $V_1$ and $V_2$ are the color classes and  $|V_1|=x$. Then $|V_2|=n-x$ and $|V_1| \leq |V_2|$. Let $F(x)$ denote the minimum number of edges that have to be removed from $K_n$ to obtain such a graph $G$. Thus $\overline{M}(n,2)$ equals the minimum value of $F(x)$ over the set $\{1, \dots, \lfloor \frac{n}{2} \rfloor\}$.

   Consider a vertex $v \in V_1$. Since $v$ has at most $x-1$ neighbors in $V_1$, the majority condition~\eqref{eq:majority-2} implies $|N(v) \cap V_2|\leq x-1$. Consequently, at least $|V_2|-(x-1) = n-2x+1$ edges incident to $v$ must be removed from $K_n$. Since this is true for every vertex in $V_1$, and these removed edges are pairwise different, $F(x) \geq x(n-2x+1)$ follows. In fact, as $|V_1| \leq |V_2|$, it is easy to find an arrangement that proves $F(x) = x(n-2x+1)$. %Consider now the function $F(x)$ on the real interval $[1, \lfloor n/2 \rfloor ]$. 
   Since 
   \begin{align*}
       F(x) &= x(n-2x+1)= -2x^2+(n+1)x\\
            &=-2\Bigl(x-\frac{n+1}{4}\Bigr)^2+ \frac{(n+1)^2}{8},
   \end{align*}
   the minimum value of $F(x)$ over $\{1, \dots, \lfloor \frac{n}{2} \rfloor \}$ is either $F(1)$ or $F(\lfloor \frac{n}{2} \rfloor)$. Observe that $F(1)=n-1$ for every $n \ge 2$.

   If $n$ is even, then
   $$
   F\left(\left\lfloor \frac{n}{2} \right\rfloor \right) = F\left(\frac{n}{2}\right) = \frac{n}{2} \leq n-1 =F(1).
   $$
    This implies $\overline{M}(n,2) = \frac{n}{2}$ and, in turn, 
    \begin{equation} \label{eq:M(n,2)-even}
      M(n,2) \leq  {n \choose 2} - \frac{n}{2} 
     %= \frac{n(n-1)}{2} - \frac{n}{2} 
     = \frac{n(n-2)}{2}.  
    \end{equation}

     If $n$ is odd, then
   $$
   F\left(\left\lfloor \frac{n}{2} \right\rfloor \right) = F\left(\frac{n-1}{2}\right) = \frac{n-1}{2} \cdot 2 = n-1 =F(1).
   $$
    Thus $\overline{M}(n,2) = n-1$ and
    \begin{equation} \label{eq:M(n,2)-odd}
        M(n,2) \leq  {n \choose 2} - (n-1) 
     %= \frac{n(n-1)}{2} - (n-1) 
     = \frac{(n-1)(n-2)}{2}
     = {n-1 \choose 2}.
    \end{equation}

     Finally, we point out that~\eqref{eq:M(n,2)-even} and \eqref{eq:M(n,2)-odd} hold with equality. It suffices to show that the removal of $\overline{M}(n,2)$ edges from $K_n$ may indeed yield a graph $G_n$ with the equality $\mc(G_n)=2$, for every $n \ge 2$.
     If $n$ is even, we remove a perfect matching from $K_n$ to obtain $G_n$. A $\cg$-coloring of $G_n$ with two colors is easy to find. Using three colors however is not possible as for the smallest color class $V_1$, a vertex $v \in V_1$ has at most $\lfloor \frac{n}{3} \rfloor-1$ neighbors in $V_1$ while $\deg_{G_n}(v) =n-2 > 2 \lfloor \frac{n}{3} \rfloor -2$, contradicting the majority condition~\eqref{eq:majority-1}. 
     If $n$ is odd, then one example is the graph $G_n$ consisting of two components: one is a complete graph of order $n-1$, the other is an isolated vertex. It is clear that $G_n$ can be obtained by removing $n-1$ edges from $K_n$ and that $\mc(G_n)=2$ holds.
   \end{proof}

Note that Theorem~\ref{thm:M(n,2)} additionally confirms Corollary~\ref{cor:Kn-e}. Moreover, Theorem~\ref{thm:M(n,2)} will serve as the base case for the induction when the following more general statement is proved. 

   \begin{theorem} \label{thm:M(n,k)}
       For every two integers $n \ge k \ge 2$, the maximum number of edges in an $n$-vertex graph $G$ with $\mc(G) =k$ is 
      \begin{equation} \label{eq:thm-M(n,k)}
    M(n,k)= 
    \begin{cases}
	{n-k+1 \choose 2 }, &\text{if $n-k$ is odd},\\
	{n-k+2 \choose 2 }- \frac{n-k+2}{2}, &\text{if $n-k$ is even}.\\
    \end{cases}
    \end{equation}
   \end{theorem}
   \begin{proof}
   First remark that, for $k=2$, these formulas are equivalent to those in Theorem~\ref{thm:M(n,2)}. Hence the statement holds for $k=2$ and any $n \ge 2$. We proceed by induction on $k$ by assuming  $k\ge 3$ and that the statement is true for the smaller values of $k$. 
   
   Let $G$ be a graph with $n$ vertices and $\mc(G) \ge k$. Let $V_1, \dots V_k$ be the (nonempty) color classes in a $\cg$-coloring of $G$ and $n_i=|V_i|$ for $i \in [k]$. We assume, without loss of generality, that $n_1 \leq n_2 \leq \cdots \leq n_k$. By $H(x)$, we denote the maximum number of edges in an $n$-vertex graph $G$ when $n_1=x$ and $\mc(G) \ge k$. Observing that $x \in \{1, \dots , \lfloor \frac{n}{k} \rfloor \}$, we distinguish the following three cases. 
   \begin{description}
       \item[Case 1. $n_1=1$] 
       By the majority condition~\eqref{eq:majority-1}, the vertex $v_1$ in $V_1$ is an isolated vertex of $G$. Therefore, $\mc(G) \ge k$ implies $\mc(G-v_1) \ge k-1$ and we can apply the hypothesis for $G-v_1$. If $n-k$ is odd, then so is $(n-1)-(k-1)$ and we obtain
       $$H(1) \leq {(n-1)- (k-1) +1 \choose 2}= {n-k+1 \choose 2}.
       $$
       If $n-k$ is even, then so is $(n-1)-(k-1)$ and therefore,
       \begin{align*}
           H(1) &\leq {(n-1)- (k-1) +2 \choose 2}- \frac{(n-1)-(k-1)+2}{2}\\
       &= {n-k+2 \choose 2 }- \frac{n-k+2}{2}.
       \end{align*}
         In both cases, the formulas in~\eqref{eq:thm-M(n,k)} are upper bounds on $H(1)$.  
       \item[Case 2. $2 \leq n_1 <\frac{n}{k}$]  
       We modify $G$ to obtain a graph $G'$ with a $\cg$-coloring that uses $k$ colors and the smallest color class contains only one vertex. We also ensure that the size of $G$ does not decrease in the process.

       Let $V_1, \dots, V_k$ be the color classes for $G$. Choose an arbitrary vertex $v_1$ from $V_1$ and let $X=V_1 \setminus \{v_1\}$. Since $n_1 \ge 2$, the cardinality of $X$, denoted by $x$, is at least $1$. To construct $G'$, delete all edges incident to $v_1$ and add all edges (that were missing in $G$) between $X$ and $V_k$. 

       We first prove that $m(G)$ and $m(G')$, the number of edges in $G$ and $G'$ respectively, satisfy $m(G) \leq m(G')$. Vertex $v_1 $ has at most $n_1-1$ neighbors in $V_1$. By the majority condition~\eqref{eq:majority-1}, $\deg_G(v_1) \leq 2n_1-2$ follows. Therefore, we deleted at most $2n_1-2$ edges from $G$. By the same majority condition, every vertex $u \in X$ has at most $n_1-1$ neighbors from $V_k$ in $G$. The condition $n_1 < \frac{n}{k}$ implies $n_k \ge n_1+1$, and consequently, at least two new edges incident to $u$ were added to obtain $G'$. Therefore, the number of new edges in $G'$ is at least $2x$, and we may derive 
       $$m(G') \ge m(G) -(2n_1-2) + 2x = m(G). 
       $$
       Next, we show that $V_1', \dots, V_k'$, where $V_1'=\{v_1\}$, $V_k'=V_k \cup X$, and $V_i'=V_i$ for $i\in \{2, \dots, k-1\}$, defines a $\cg$-coloring for $G'$. The majority condition~\eqref{eq:majority-2} clearly holds for the isolated vertex $v_1$. For the vertices in $\bigcup_{i=2}^{k-1} V_i$, the number of neighbors from the same color class remained unchanged, and the neighbors from different classes remained the same or decreased by $1$, and~\eqref{eq:majority-2} still holds in $G'$. For every vertex $u \in X$, since the majority condition holds for $u$ in $G$, and further by the construction of $G'$, we have that
       \begin{align*}
           |N_{G'}(u) \cap V_k'| &\ge |N_{G}(u) \cap V_1 | \ge \Bigl|N_{G}(u) \cap \bigcup_{i=2}^{k-1} V_i \Bigr|\\
           &= \Bigl|N_{G'}(u) \cap \bigcup_{i=2}^{k-1} V_i' \Bigr|
       = \Bigl|N_{G'}(u) \cap \bigcup_{i=1}^{k-1} V_i' \Bigr|.
       \end{align*}
       That is, $u$ satisfies the condition~\eqref{eq:majority-2} in $G'$. The last case to check is when $u \in V_k$. During the construction of $G'$, vertex $u$ might gain some new neighbors from $X \subseteq V_k'$, and it might lose one neighbor, namely $v_1$, which is outside $V_k'$. It implies that~\eqref{eq:majority-2} remains true for $u$ in $G'$. We conclude that $\mc(G') \ge k$ and there is a $\cg$-coloring of $G'$ where the smallest color class consists of an isolated vertex.

       Consequently, for every graph $G$ satisfying the conditions of this case, we can construct a graph $G'$ with $m(G') \ge m(G)$ such that a $\cg$-coloring of $G'$ belongs to Case~1. Formally,
       $$m(G) \leq m(G') \leq H(1)$$
       and hence, the relevant formula in~\eqref{eq:thm-M(n,k)} gives an upper bound on the size of $G$.
       \item[Case 3. $n_1=\frac{n}{k}$]
       In this case, the color classes are of the same cardinality $\frac{n}{k}$. Given a graph $G$ that admits such a $\cg$-coloring, we construct $G'$ as in Case~2. The only difference is that instead of always choosing the color class $V_k$ to be replaced with $V_k \cup X$, we choose a color class $V_s$ with $s \ge 2$ such that the number of edges between $V_s$ and $X$ is the smallest in $G$. 
       
       By the majority condition~\eqref{eq:majority-2}, there are at most $(n_1-1)^2$ edges between $X$ and $\bigcup_{i=2}^k V_i$ in $G$. Since $k \ge 3$, the number of edges between $X$ and $V_s$ is at most $\lfloor \frac{1}{2} (n_1-1)^2 \rfloor$. When constructing $G'$, the at most $2n_1-2$ edges incident to $v_1$ are removed, and at least
       $(n_1-1)|V_s|- \lfloor \frac{1}{2} (n_1-1)^2 \rfloor$ new edges are added between $X$ and $V_s$. If $n_1 \ge 3$, then the size of $G'$ can be estimated as 
       \begin{align} \label{eq:case-3}
                m(G') &\ge m(G) -(2n_1-2) + (n_1-1)n_1 - \Bigl\lfloor \frac{(n_1-1)^2}{2} \Bigr\rfloor\\
        \nonumber  &\ge m(G) + \frac{n_1^2}{2}-2n_1+\frac{3}{2} \ge m(G).
         %&=  m(G) + \frac{1}{2}(n_1-2)^2-0.5          
       \end{align} 
       If $n_1=2$, then $\lfloor \frac{1}{2} (n_1-1)^2 \rfloor=0$ and a direct substitution into~\eqref{eq:case-3} shows $m(G') \ge m(G)$. 

       What remains to check is that the partition into $V_1'=\{v_1\}$, $V_s'= V_s \cup X$, and $V_i'=V_i$ for $i \in [k] \setminus \{1,s\}$ defines a $\cg$-coloring of $G'$. It can be done along the same lines as for Case~2.
       We may infer again that $H(1)$, and consequently also the relevant formula from~\eqref{eq:thm-M(n,k)}, give an upper bound on $m(G)$.
   \end{description}
   The three cases above cover all possibilities. We therefore conclude that~\eqref{eq:thm-M(n,k)} gives an upper bound on the number of edges for every graph $G$ with  $\mc(G) \ge k$. 
   \medskip
   
   To complete the proof, we show that the upper bound is attained by a graph $G_{n,k}$ for every $n \ge k$ such that $\mc(G_{n,k})=k$ also holds. If $n-k$ is odd, let $G_{n,k}$ be the graph consisting of $k-1$ isolated vertices and a component which is a complete graph on $n-k+1$ vertices.  If $n-k$ is even, we take $k-2$ isolated vertices and a complete graph of order $n-k+2$ from which a perfect matching is removed to obtain $G_{n,k}$. It is straightforward to check that $\mc(G_{n,k})=k$ holds in both cases and the size of  $G_{n,k}$ corresponds to the values in~\eqref{eq:thm-M(n,k)}. 
   \end{proof}

   We conclude this subsection by identifying the minimum possible number of edges in a graph with a given order and $\cg$-chromatic number.

   \begin{proposition}  \label{prop:m(n,k)}
       For every two integers $n \ge k \ge 1$, the minimum number of edges in an $n$-vertex graph $G$ with $\mc(G) =k$ is
       $$ m(n,k)= n-k.
       $$
   \end{proposition} 
      \begin{proof}
       We can always obtain a majority C-coloring of a graph $ G$ by assigning different colors to its components. Therefore, $\mc(G)=k$ does not allow $G$ having more than $k$ components. Thus $G$ must contain at least $n-k$ edges, and $m(n,k) \ge n-k$.  On the other hand, for any two integers $n$ and $k$ with $n \ge k \ge 1$, we can construct a graph $F_{n,k}$ that contains a star on $n-k+1$ vertices and further $k-1$ isolated vertices. By Observation~\ref{obs:simple_graph_classes}(iii), the $\cg$-chromatic number of a star  is~$1$. Then Observation~\ref{obs:simple_graph_classes}(iv) implies $\mc(F_{n,k})=k$.  As $F_{n,k}$ contains $n-k$ edges, $m(n,k) \le n-k$ follows, and we may conclude $m(n,k)=n-k$.
   \end{proof}

   %%%%%%%%%%%%%%%%%%%%%%%%%%%
   %\section{Chromatic number and $\cg$-chromatic number}  \label{sec:chi-and-mc}
   %\cs{It is a very short section, we should move it somewhere. If we don't find a better place, it can be a subsection in a concluding remarks section. Or as a subsection at the end of Sec 3?}
   %\hf{Good idea. Maybe in 2.1 General upper bounds?} \cs{2.1 would be too soon. Proposition~\ref{prop:chi-mc} uses Prop~\ref{prop:cycle-k} and Theorem~\ref{thm:chi-mc-n} uses Theorem~\ref{thm:M(n,k)}.}
\subsection{Order, $\chi(G)$, and $\mc(G)$}
   We first show that the chromatic number and the $\cg$-chromatic number of graphs are incomparable. In contrast to this fact, we prove that the results in Section~\ref{subsec:size-order-mc} imply a sharp inequality between $\chi(G)$, $\mc(G)$, and $n(G)$, where $n(G)$ denotes the order of $G$. 
   \begin{proposition} \label{prop:chi-mc}
       For every two positive integers $k_1$ and $k_2$, there exists a graph $G$ with $\chi(G)=k_1$ and $\mc(G)=k_2$.
   \end{proposition}
   \begin{proof}
       If $k_1=1$, we take $k_2$ isolated vertices as a graph $G$, and observe that $\chi(G)=1$ and $\mc(G)=k_2$. If $k_2=1$ then $G= K_{k_1}$ satisfies the requirements. For $k_1, k_2 \ge 2$, we define $G$ as a power of a path. Let $G= P_n^k$ with $k=k_1-1$ and $n=k_1k_2$. Since 
       %As $G$ is a perfect graph and 
       $\omega(G)=k+1=k_1$, and a  classical vertex coloring of $G$ with $k_1$ colors is easy to obtain, we have $\chi(G)=k_1$.  
       On the other hand, Theorem~\ref{thm:power-paths} implies $\mc(G)= \lfloor \frac{n}{k+1} \rfloor = k_2 $ that verifies the statement.
   \end{proof}

As an immediate consequence of Proposition~\ref{prop:chi-mc}, the parameters 
$\chi(G)$ and $\mc(G)$ are also independent in the sense that the difference $\mc(G)-\chi(G)$ can take arbitrarily large positive as well as arbitrarily small negative values.
   
   \begin{theorem} \label{thm:chi-mc-n}
       If $G$ is a graph of order $n$, then $\chi(G)+ \mc(G) \leq n+1$. Further, the inequality is sharp.
   \end{theorem}
   \begin{proof}
       Let $G$ be an arbitrary graph with $n$ vertices and $m$ edges. Let $\chi(G)=k$ and consider the color classes defined by a $k$-coloring of $G$. Since $k$ is the minimum number of colors, no two color classes can be merged, and therefore $m \ge {k \choose 2}$.

       Suppose now, contrary to the statement, that $\mc(G) \ge n+2-k$.  Theorem~\ref{thm:M(n,k)} implies that the strict inequality
       $$m < {n-(n+2-k)+2 \choose 2} ={k \choose 2}
       $$ 
 holds regardless of the parity of $n-k$. Since it contradicts the conclusion derived from $\chi(G)=k$, we may infer $\mc(G) \leq n +1 -k$, which finishes the proof of the inequality. 
       
       The sharpness can be shown by taking a complete graph $K_n$ and observing $\chi(K_n)+ \mc(K_n) = n+1$; or by taking an edgeless graph $G$ on $n$ vertices and observing $\chi(G)=1$ and $\mc(G)=n$. 
   \end{proof}

%%%%%%%%%%%%%%%%%%%%%%%%%%%%%%%%%%%%%%%%
\section{Linear-time algorithm for trees}
\label{sec:sec-tree}
By Proposition~\ref{prop:NP-c} the problem of deciding whether $\mc(G)\geq k$ is NP-complete for any fixed $k\geq 2$. 
In this section, we first show that majority 
C-colorings are closely related to edge cuts under certain conditions. We then investigate the time complexity of determining $\cg$-chromatic number for trees. In Section~\ref{sec:tree} we present a linear-time algorithm that outputs a $\mc$-coloring for any tree.

\subsection{Cut subgraphs with bounded degree}

\begin{definition}
     Given a graph $G = (V,E)$ and a subset $F$ of its edges, we denote by $[F]$ the subgraph of $G$ induced by $F$; i.e., $[F]$ has vertex set $\bigcup_{e\in F} e$ and edge set $F$. 
     We say that $[F]$ is a \emph{cut-$s$ subgraph} of $G$ if the graph $G^* = (V,E \setminus F)$ has at least $s$ components, $s \ge 1$.  
     Moreover, if a cut-$s$ subgraph $[F]$ has such a property that for each vertex $x \in V([F])$ its degree in $[F]$ is at most $ \frac{1}{2}\deg_G(x)$, it will be referred to as a \emph{cut-$(s,\leqslant)$ subgraph} of $G$.
\end{definition}

\begin{proposition}
    For every graph $G$ the following properties are equivalent:
    \begin{enumerate}[label=\normalfont(\roman*)]
        \item $\mc (G) \geq s$;
        \item $G$ has a cut-$(s,\leqslant)$ subgraph.
    \end{enumerate}
    \label{prop:cutdeg}
\end{proposition}
\begin{proof}
    Suppose that $\mc(G) \ge s$. Then there exists a majority C-coloring
$\phi$ of $G=(V,E)$ that uses exactly $s$ colors.
By the definition of a majority C-coloring, for every vertex $v \in V$, at least half
of its neighbors receive the same color as $v$. Equivalently, at most half of the
neighbors of $v$ are colored differently from $\phi(v)$.
Let $E' \subseteq E$ denote the set of non-monochromatic edges, that is, an edge
$e=xy \in E$ belongs to $E'$ if and only if $\phi(x) \neq \phi(y)$.
Since each vertex $x \in V$ has at most $\frac{1}{2}\deg_G(x)$ neighbors with color different
from $\phi(x)$, it follows that
\[
\deg_{[E']}(x) \le \frac{\deg_G(x)}{2}
\quad \text{for every } x \in V.
\]
Observe that every edge in $E \setminus E'$ connects two vertices of the same color.
Hence, each component of the graph $G^* = (V, E \setminus E')$
is monochromatic.
Since $\phi$ uses exactly $s$ colors, the graph $G^*$ has at least $s$ components.
Consequently, the edge set $E'$ forms a cut-$(s,\leqslant)$ subgraph of $G$.

Now suppose that the graph $G=(V,E)$ admits a cut-$(s,\leqslant)$ subgraph induced by an edge
set $F \subseteq E$.
Define a vertex coloring $\phi$ by assigning the same color to two vertices if and
only if they belong to the same component of the graph $G^* = (V, E \setminus F)$.
By the definition of a cut-$(s,\leqslant)$ subgraph, the graph $G^*$ has at least $s$ 
components, and hence the coloring $\phi$ uses at least $s$ colors.
On the other hand, if two adjacent vertices $x$ and $y$ receive different colors under
$\phi$, then $xy \in F$.
Since $\deg_{[F]}(x) \le \frac{1}{2}\deg_G(x)$ for every vertex $x \in V$, it follows that $x$
has at most half of its neighbors colored differently from $\phi(x)$.
Consequently, $\phi$ is a majority C-coloring of $G$ using at least $s$ colors,
which completes the proof.

\end{proof}
In view of the proof, a cut-$(s,\leqslant)$ subgraph corresponds precisely to a decomposition of $G$ into at least 
$s$ monochromatic components under the associated majority C-coloring. %In the remainder of this section, we use these two notions interchangeably whenever convenient.

\subsection{Linear-time algorithm for trees}\label{sec:tree}
Next we describe a linear-time algorithm that computes the $\cg$-chromatic number and outputs a $\mc$-coloring for trees.

\begin{theorem}
    The $\cg$-chromatic number and a $\mc$-coloring can be determined in linear-time for trees. 
\end{theorem}
\begin{proof}
     Let $T$ be a tree. According to Proposition \ref{prop:cutdeg}, $\mc(T)$ is equal to the largest  $s$ such that $T$ has a cut-$(s,\leqslant)$ subgraph. First, we describe a procedure
 which calculates this largest number $s$ for trees. 
 %This value corresponds to the largest number of colors 
 
Consider a tree $T = (V,E)$ with a fixed root vertex $v$. For every vertex $x \in V$, $T(x)$  denotes the subtree rooted at $x$.  The parent vertex of $x$ is denoted by $p(x)$.   
The following functions $f$, $g$, and $h$ are defined for every $x \in V$, excluding leaves and the root of $T$. 
%If $x$ is not the root of $T$, we have as follows.

 \begin{itemize}
     \item $f(x)$: the largest number of components which can be obtained from $T(x)$ by deleting a cut subgraph of the following property: for each vertex $z \in V(T(x))$ its degree in this cut subgraph is at most $ \frac{1}{2}\deg_{T}(z)$, and $x$ has degree at most $\frac{1}{2}(\deg_{T(x)}(x)+1)$ in this cut subgraph. 

     %As the consequence, to ensure the appropriate number of neighbors of $x$ in the same color as $x$, in the relevant majority C-coloring of the entire tree $T$, we will need to ensure the same color for $p(x)$ as for $x$.

\item $g(x)$: the largest number of components which can be obtained from $T(x)$ by deleting a cut subgraph of the following property: for each vertex $z \in V(T(x))$ its degree in this cut subgraph is at most $ \frac{1}{2}\deg_{T}(z) $, and $x$ has degree at most $\frac{1}{2}(\deg_{T(x)}(x)-1)$ in this cut subgraph. 

%In this case, in the relevant majority C-coloring of $T$, the number of neighbors of $x$ in the same color as $x$ is already ensured in $T(x)$, independently of the color assigned to $p(x)$.

\item $h(x) = g(x) - f(x) + 1$.
 \end{itemize}
It is obvious by definition that $f(x) \geq g(x)$ holds for every $x \in V(T)$ for which these functions are defined, and therefore $h(x) \leq 1$.

Consider a cut subgraph $S$ of $T(x)$ in which $x$ has degree at most $\frac{1}{2}(\deg_{T(x)}(x)+1)$ and for each vertex $z \in V(T(x))$ its degree in this cut subgraph is at most $\frac{1}{2}\deg_{T}(z)$. Note that in case when all the children of $x$ are leaves, such a cut subgraph is empty and then $f(x) = g(x)$. If we delete $S$ from $T(x)$, we obtain $f(x)$ components. If we remove an edge $e$ incident to $x$ from $S$, we obtain a cut subgraph $S'= S  \backslash \{e\}$ where $x$ has a degree at most $\frac{1}{2}(\deg_{T(x)}(x)-1)$. Deleting $S'$ from $T(x)$ yields $f(x)-1$ components. This shows that $g(x) \geq f(x)-1$ and therefore $h(x) \in \{0,1\}$.

The values $f(x)$, $g(x)$ and $h(x)$ will be determined in postorder. 
If a vertex $x$ has only adjacent descendants that are leaves, then clearly, $f(x)=g(x)=1$, and $h(x)=1$.
Otherwise, assume that $x$ has $d$ children not being leaves, say $x_1,\ldots,x_d$. 
Recall that a cut-$(s \leqslant)$ subgraph, by definition, cannot contain an edge incident to a leaf.
We will take this fact into account in the subsequent calculations.
  \begin{itemize}
      \item If $x$ has at most $\frac{1}{2}(\deg_{T(x)}(x)+1)-1$ children not being leaves, then $f(x)=g(x)=1+\sum_{i=1}^d \max \{f(x_i)-1,g(x_i)\}=1+\sum_{i=1}^d g(x_i)$. Hence, $h(x)=1$ in such a case.
      \item If $x$ has more than $\frac{1}{2}(\deg_{T(x)}(x)+1)-1$ children not being leaves, i.e., if $d \geq \frac{1}{2}(\deg_{T(x)}(x)-1)$, then  we can suppose that the children of $x$ of degree at least 2 in $T$ are ordered according to decreasing $h$-value; i.e., $h(x_1) \geq h(x_2) \geq \cdots \geq h(x_d)$ (the adjacent leaves are not taken into account in the calculations). Let $c=\lfloor\frac{1}{2}(\deg_{T(x)}(x)+1)\rfloor$. Then we have
      $$f(x)=1+\sum_{i=1}^c g(x_i) + \sum_{i=c+1}^d (f(x_i)-1),$$
      and analogously
      $$g(x)=1+\sum_{i=1}^{c-1} g(x_i) +\sum_{i=c}^d (f(x_i)-1).$$
  \end{itemize}
Now, let us assume that the functions $f$, $g$, and $h$ are calculated for every $x \in V(T)$, excluding leaves and the root of $T$. If $x$ is the root of $T$, i.e. if $x=v$, we calculate the largest number of components which can be obtained from $T$ by deleting a cut subgraph of property: for each vertex $z \in V(T)$ its degree in this cut subgraph is at most $ \frac{1}{2}\deg_{T}(z)$, i.e. the value of $s$.
To calculate $s$ we use the following formula. Let $c'=\lfloor \frac{1}{2}\deg_T(v)\rfloor$.
$$
s=1+\sum_{i=1}^{c'} g(x_i) + \sum_{i=c'+1}^d (f(x_i)-1).$$

Since $h(x_i) \in \{0, 1\}$ for all $1 \leq i \leq d$, sorting these numbers takes $O(d)$ time. The computation of $f(x)$ and $g(x)$ can be done in time proportional to the number of children of $x$ not being leaves.  Thus, determining the value of $\mc(T) = s$ takes linear time for trees.

The corresponding cut-$(s,\leqslant)$ subgraph $S$ can be obtained by traversing the tree in preorder. We keep the notation for the children of $x$ not being leaves as $x_1,x_2,\ldots,x_d$ such that $h(x_1) \geq  h(x_2) \geq \cdots
\geq h(x_d)$ holds. Vertices that are leaves are not considered in this procedure; by definition, the edges between leaves and their parents are not part of the cut-$(s,\leqslant)$ subgraph $S$.
\begin{enumerate}
    \item If $x=v$ or $p(x)x \in S$ then $xx_i \in S$ if and only if $i \leq \frac{1}{2}\deg_T(x)$.
    \item If $p(x)x \in S$ then $xx_i \in S$ if and only if $i \leq \frac{1}{2}\deg_T(x)-1$.
\end{enumerate}

A corresponding $\mc$-coloring of tree $T$ can be obtained based on the proof of Proposition \ref{prop:cutdeg}, which completes the proof.
\end{proof}

\section{Sufficient conditions for $\mc(G)\ge 2$}  \label{sec:sufficient}

One of the most important questions of the satisfactory graph partition problem is to decide whether a graph $G$ is partitionable that is, with our terminology, whether $\mc(G) \ge 2$ holds. Recall that, by Proposition~\ref{prop:NP-c}, it is a hard problem. %By Proposition~\ref{prop:cycle-k}, we see that every $2$-regular graph except $C_3$ satisfies the inequality. 
The following sufficient conditions for the relation $\mc(G) \ge 2$ to hold are known from~\cite{BTV-2003, GeKob-2004, Moshi, Shaf}: (i) $G$ is a cubic graph different from $K_4$ and $K_{3,3}$; (ii) $G$ is $4$-regular and different from $K_5$; (iii)  $G$ is not a star and its girth is at least $5$.
Our goal in this section is to establish further sufficient conditions. 

\begin{theorem} \label{thm:partitionable} \label{thm:girth}
Let $G$ be a graph with maximum degree $\Delta(G)$, clique number $\omega(G)$, and girth $g(G)$.
\begin{enumerate}
[label=\normalfont(\roman*)]
    \item Let $\Delta(G)=k$. If $\omega(G) \geq \lceil \frac{k}{2}\rceil+1$ and the order of $G$ is greater than $(\lceil\frac{k}{2}\rceil +1)(\lfloor \frac{k}{2}\rfloor +1)$, then $\mc(G)\geq 2$.
     \item If $\Delta(G) \leq 4$ and the order of $G$ is greater than $3g(G)$, then $\mc(G)\geq 2$.
\end{enumerate}  
\end{theorem}
\begin{proof}
    (i) Let $G=(V,E)$ be a graph satisfying the conditions of the statement. We begin constructing a $\cg$-coloring of $G$ with two color classes by defining a set $X_1$ containing the vertices of a complete subgraph of order $\lceil \frac{k}{2} \rceil+1$.
    Let us denote the number of edges between $X_1$ and $V\setminus X_1$ by $m_1$. 
    Since every vertex in $X_1$ has exactly $\lceil \frac{k}{2} \rceil$ neighbors in $X_1$ and at most $k-\lceil \frac{k}{2} \rceil = \lfloor \frac{k}{2} \rfloor$ neighbors in $V \setminus X_1$,  
    we may infer $m_1 \leq (\lceil\frac{k}{2}\rceil +1)\lfloor \frac{k}{2}\rfloor$.
    If there exists a vertex $v_1 \in V\setminus X_1$ that does not satisfy the majority condition~\eqref{eq:majority-2}; that is, if 
    \begin{equation} \label{eq:necessary-algo}
        |N(v_i) \cap X_i| > |N(v_i) \cap (V \setminus X_i)|
    \end{equation}
    holds with $i=1$, 
    then we move $v_1$ to $X_1$. Formally, we set $X_2=X_1 \cup \{v_1\}$. Let $m_2$ be the number of edges between $X_2$ and $V\setminus X_2$. It follows from~\eqref{eq:necessary-algo} and the construction of $X_2$ that
    \begin{equation*} %\label{eq:necessary-2}
       m_2= m_1 - |N(v_1) \cap X_1| + |N(v_1) \cap (V \setminus X_1)| \leq m_1-1. 
    \end{equation*}
    By the definition of $X_1$, every vertex $u \in X_1$ satisfies the majority condition $|N(u) \cap X_1| \ge \frac{1}{2} \deg_G(u)$, and it remains true for $N(u) \cap X_2$. By the definition of $X_2$ and~\eqref{eq:necessary-algo}, the inequality $|N(v_1) \cap X_2| \ge \frac{1}{2} \deg_G(v_1)$ is also true.
    
    We repeat this procedure and choose a vertex $v_i$ in the $i$th step which satisfies~\eqref{eq:necessary-algo}, while it is possible, and define $X_{i+1}=X_i \cup \{v_i\}$. Analogously to the first step, for the number of edges between $X_{i+1}$ and its complement, we have $m_{i+1} \leq m_i -1$. Therefore, $m_{i+1} \leq m_1 -i$. It also remains true that $|N(u) \cap X_{i+1}| \ge \frac{1}{2} \deg_G(u)$ holds for each $u \in X_{i+1}$. %Note that $|X_{i+1}|= \lceil k/2 \rceil+1+i $.

    Since $m_{i+1} \leq m_1 -i$ for every index $i$, after at most $m_1$ steps, we obtain a set $X_{j}$ such that no vertex $v_{j} \in V\setminus X_j$ satisfies~\eqref{eq:necessary-algo}. Equivalently, every vertex from $V \setminus X_j$ complies with the majority condition for the partition $\{X_j,V \setminus X_j\} $. As observed earlier, the same is true for the vertices in $X_j$.
    
    What remains to prove is $V \setminus X_j \neq \emptyset$. Indeed,  we have 
    $$|X_j| \leq \Bigl\lceil\frac{k}{2} \Bigr\rceil +1+m_1\leq \Bigl\lceil\frac{k}{2} \Bigr\rceil + 1 +  \Bigl(\Bigl\lceil\frac{k}{2} \Bigr\rceil+1 \Bigr)\Bigl\lfloor \frac{k}{2} \Bigr\rfloor
    =\Bigl(\Bigl\lceil\frac{k}{2} \Bigr\rceil+1 \Bigr) \Bigl(\Bigl\lfloor\frac{k}{2} \Bigr\rfloor+1 \Bigr).
    $$ 
Hence, by our assumption, $|X_j|$ is smaller than the order of $G$, and $V \setminus X_j \neq \emptyset$. 
The partition $\{X_j,V \setminus X_j\} $ therefore defines a $\cg$-coloring and the proof is complete for (i).
\medskip

(ii) We proceed along the same lines as in the proof of (i). The only difference is that at the beginning, we choose a cycle subgraph $C$ of order $g(G)$, and define $X_1$ as the vertex set of $C$. Since $\Delta(G) \leq 4$, the majority condition holds for each vertex in $X_1$. As for the number of edges between $X_1$ and $V \setminus X_1$, it holds that $m_1 \leq 2g(G)$. We construct the sets $X_2, \dots , X_j$ such that for every $i \ge 1$, we choose a vertex $v_{i}$ that satisfies~\eqref{eq:necessary-algo} and define $X_{i+1}= X_i \cup \{v_i\}$. Again, the procedure ends after at most $m_1$ steps, and we obtain a set $X_j$ so that every vertex satisfies the majority condition for the partition $\{X_j, V \setminus X_j\} $. As $|X_j| \leq |X_1| + m_1 \leq 3g(G) $, we may conclude $|V \setminus X_j| \ge 1$, and the proof is complete.
\end{proof}

By refining the counting argument used in the proof of Theorem~\ref{thm:partitionable}(ii) and restricting attention to $4$-regular graphs, we obtain the following result.

\begin{proposition}
\label{cor:4reg}
       Let $k$ be  a positive integer and let $G$ be a 4-regular graph with girth $g(G)$. If the order of the graph $G$ is at least $2g(G)+1$ then $\mc(G)\geq 2$.
       %\hfill $\Box$
    
\end{proposition}

By the Moore bound (see for example \cite{erdos_moore}), any $4$-regular graph $G$ of girth $g(G)$ has at least
$2g(G)+1$ vertices, except $K_5$ and $K_{4,4}$. As the case for $K_{4,4}$ is clear (cf. Fig.~\ref{fig:ex_substar_cubic}(b)), we may say
that our results are consistent with \cite[Proposition 3]{BTV-06}, in the language of satisfactory partitions and 
%the corresponding result obtained in \cite{BTV-06} in the language of satisfactory partitions. 
our approach 
%from the proof of Theorem~\ref{thm:partitionable}(ii) 
provides an alternative proof for that result.

\section{Concluding discussions}  \label{sec:final}
In this paper, we introduced the majority C-coloring and initiated a research on this topic. To close the paper, we discuss some possible applications of this new model and also pose two natural directions, namely the $\mc$-edge-stability and $\mc$-edge-criticality, for further research. 
\subsection{Possible applications of majority C-colorings} \label{applications}

Majority C-colorings provide a natural framework for modeling systems in which local decisions must be consistent with prevailing choices in the neighborhood. Such situations arise in networked settings where coordination between adjacent units is essential, including spatial planning, public service allocation, infrastructure coordination, and models of social or organizational behavior. In this framework, vertices represent interacting units, edges encode dependencies, and colors correspond to alternative configurations or policies. The majority condition ensures local stability by preventing isolated or conflicting assignments, while the parameter $\mc(G)$ measures the maximum diversity of configurations compatible with this form of local consistency.

A particularly natural application arises in the organization of urban services. Service zones are modeled as vertices, adjacency reflects infrastructural or operational dependencies, and colors represent different service providers, schedules, or standards. The majority condition guarantees that each zone is coordinated primarily with neighboring zones operating under the same configuration, leading to locally stable service allocation. In this context, $\mc(G)$ quantifies how many distinct service configurations can be deployed simultaneously without fragmenting the system or increasing coordination costs. For example, in waste collection or street maintenance, $\mc(G)$ indicates how many different schedules or operators can be introduced without fragmenting the service system and increasing coordination complexity.

\paragraph{Relation to Schelling's model of segregation}
Schelling's model of segregation is an agent-based model developed by economist Thomas Schelling in~\cite{schelling}. Originally, on a graph of locations, agents of two types occupy vertices, leaving some vacant. An agent is ``unhappy" if the fraction of same-type neighbors falls below a tolerance threshold $\tau\in [0,1]$. Unhappy agents move to vacant nodes, and these repeated individual shifts eventually produce large-scale segregated patterns. It is worth to mention that every agent is not guaranteed to be satisfied and in these cases, it is of interest to study the patterns (if any) of the agent dynamics. 
Although Schelling's model does not include external factors that force agents to segregate, it demonstrates how even slight in-group preferences can lead to extreme social segregation.
Even if individuals do not mind diversity, a small preference for neighbors of the same ``type'' (race or economic status) triggers a chain reaction of self-segregation.  

By integrating Schelling's model with majority C-coloring, we can replace moving with repainting. In our model every node is occupied (no vacancies). Agents evaluate their neighborhood. If the fraction of same-type neighbors is less than $\tau$, rather than move they change their type (color) to the majority type among neighbors. 
So a majority C-coloring  corresponds to a situation when $\tau=0.5$ and everybody is satisfied. Considering a graph and its arbitrary coloring with 2 colors, we are able to recolor its vertices one by one in such a way that every vertex becomes satisfied at the end, what means that the obtained coloring is a majority C-coloring.
    
    Beyond thresholds, Schelling's model can be refined using utility functions. In a social network context, this often manifests as edge rewiring rather than movement or color changes.
The agent neither changes its location nor its type (color); instead, it modifies its neighborhood by severing ties with dissimilar agents and forming new connections with similar ones. 
From a graph theory perspective, the goal is to transform an initial graph through a sequence of edge removals and additions until the existing coloring becomes a majority C-coloring.

\subsection{$\mc$-edge-stabilty and $\mc$-edge-criticality }
In analogy with classical graph coloring, it is natural to consider the behaviour of the parameter $\mc(G)$ under edge deletions (cf.~\cite{edge-stab}). In what follows, we address this question from two perspectives, which we refer to as $\mc$-edge-stability and $\mc$-edge-criticality. These two concepts naturally suggest further directions for the study of majority C-colorings.

\paragraph{$\mc$-edge-stabilty.} The stability of the $\cg$-chromatic number under edge deletion can be quantified by considering the minimum number of edges whose removal changes the value of $\mc(G)$ ($es_{\mc}(G)$). Such stability parameters, commonly referred to in the literature as edge stability, measure the robustness of graph invariants with respect to edge deletions. They provide a natural direction for further investigation in the context of majority C-colorings. 
This notion captures how robust the parameter 
$\mc(G)$
is under structural perturbations of the graph, and serves as a natural analogue of edge-stability concepts studied for classical coloring invariants. This means that deleting any $es_{\mc}(G)-1$ edges from graph $G$ leaves $\mc(G)$ unchanged.

As an illustrative example, consider the path $P_n$. If $n$ is odd, then $es_{\mc}(P_n)=1$, since deleting a single pendant edge decreases the value of $\mc(P_n)$. 
On the other hand, if $n$ is even, then the removal of any single edge does not change $\mc(P_n)$, and hence $es_{\mc}(P_n)=2$. In this case, removing two edges is sufficient to alter the value of the parameter.

As we already shown in Proposition~\ref{ineq:del_edge}, after deleting an edge from graph $G$ the majority C-chromatic number, $\mc(G-e)$,  can decrease or increase. If we consider deleting more edges, the same rule applies. 

\paragraph{$\mc$-edge-criticality.}
It is natural to ask which classes of graphs have the property that the deletion of any edge necessarily changes the $\cg$-chromatic number. Such graphs with at least one edge will be called $\mc$-\emph{edge-critical graphs}. In other words, $\mc(G)\neq \mc(G-e)$ holds for every edge $e$ if $G$ is $\mc$-edge-critical. 

As an example, consider subdivided stars $S(S_n)$ with odd $n$ (cf. Observation~\ref{obs:simple_graph_classes}(iii). In this case, the deletion of any edge increases the value of $\mc(G)$ by one, and hence such graphs are $\mc$-edge-critical. This property, however, does not hold for subdivided stars with even $n$, where there exist edges whose removal does not affect the value of $\mc(G)$.

Since the deletion of an edge may either increase or decrease the $\cg$-chromatic number (cf. Proposition~\ref{ineq:del_edge}), it is natural to ask whether both cases may occur for $\mc$-edge-critical graphs. The example in the previous paragraph, that is a subdivided star $S(S_n)$ for an odd $n$, shows that there are $\mc$-edge-critical graphs satisfying $\mc(G-e)= \mc(G)+1$ for every $e \in E(G)$. The following proposition shows that the reverse inequality $\mc(G-e) < \mc(G)$ cannot hold for every edge $e$ of a graph $G$.
\begin{proposition} \label{prop:critical}
Every graph $G$ with $E(G)\neq \emptyset$ contains an edge $e$ such that $\mc(G-e) \ge \mc(G)$.  
\end{proposition}
\begin{proof}
  It suffices to consider connected graphs. Further, if $\mc(G)=1$, then $\mc(G-e) \ge \mc(G)$ clearly holds for all $e \in E(G)$. Assume now that $\mc(G) \ge 2$ and let $\phi$ be a $\cg$-coloring of $G$ that uses $\mc(G)$ colors. Since $G$ is connected, there exists an edge $e=uv$ such that $\phi(u) \neq \phi(v)$. We claim that $\phi$ remains a $\cg$-coloring in $G-e$ and therefore, $\mc(G-e) \ge \mc(G)$. Indeed, for a vertex $z \in V(G-e)$ and the color class $V_i$ containing $z$, the majority condition~\eqref{eq:majority-1} clearly remains true if $z \notin \{u,v\}$. If $z\in \{u,v\}$, then $N_{G-e}(z) \cap V_i= N_G(z) \cap V_i$ and $\deg_{G-e}(z) = \deg_{G}(z)-1$ ensure that~\eqref{eq:majority-1} is satisfied. 
\end{proof}

By Proposition~\ref{prop:critical}, there are no $\mc$-edge critical graphs satisfying $\mc(G) > \mc(G-e)$ for all $e\in V(G)$. What is more, we conjecture that $\mc(G-e) = \mc(G)+1$ holds for every edge $e \in E(G)$ if $G$ is an $\mc$-edge critical graph.

\begin{conj}\label{conj:criticality}
    Let $G$ be a $\mc$-edge-critical graph with at least two edges and let $e,f \in E(G)$. It is no possible that $\mc(G-e) < \mc(G) < \mc(G-f)$. 
\end{conj}

Our previous considerations provide evidence supporting this conjecture. In particular, it holds for subdivided stars $S(S_n)$ with odd $n$. Moreover, it is easy to verify that odd cycles satisfy the conjecture as well.
\bigskip

The concepts of $\mc$-edge-stability and $\mc$-edge-criticality open up a range of natural questions regarding the behavior of $\mc(G)$ under edge deletions. We conclude by formulating several problems and a conjecture in this direction.

\begin{open}
Determine or estimate the parameter $es_{\mc}(G)$ for various classes of graphs.
\end{open}

\begin{open}
Characterize graphs for which $es_{\mc}(G)=1$.
\end{open}

\begin{open}
    Which classes of graphs are $\mc$-edge-critical?\label{open:crit}
\end{open}

If Conjecture~\ref{conj:criticality} is not true then the following open problem is different from Open problem \ref{open:crit} and it is worth studying separately.
%Since the deletion of an edge may either increase or decrease the $\cg$-chromatic number (cf. Proposition~\ref{ineq:del_edge}), it is natural to distinguish between these two phenomena and study them separately. 

\begin{open}
    Which classes of graphs satisfy $\mc(G) < \mc(G-e)$ for every edge $e\in E(G)$?
\end{open}

\section*{Acknowledgements}
Cs.\ Bujt\'as was supported by the Slovenian Research and Innovation Agency (ARIS) under the grants P1-0297, N1-0355, and J1-70045. 
M.\ Dettlaff and H.\ Furma\'nczyk were supported by the European Commission's Horizon Europe Research and Innovation programme through the Marie Skłodowska-Curie Actions Staff Exchanges (MSCA-SE) under Grant Agreement no.101182819 (COVER: (C)ombinatorial (O)ptimization for (V)ersatile Applications to (E)merging u(R)ban Problems).

\end{document}